\documentclass{article}
\newif\ifneuripsfinal
\neuripsfinaltrue
\newif\ifneuripspreprint
\neuripspreprintfalse
\ifneuripsfinal
  \usepackage[main, final]{Formatting_Instructions_For_NeurIPS_2026/neurips_2026}
\else
  \ifneuripspreprint
    \usepackage[preprint]{Formatting_Instructions_For_NeurIPS_2026/neurips_2026}
  \else
    \usepackage{Formatting_Instructions_For_NeurIPS_2026/neurips_2026}
  \fi
\fi
\usepackage{amsmath,amsthm,amssymb}
\usepackage{graphicx}
\usepackage{booktabs}
\usepackage{tabularx}
\usepackage{algorithm}
\usepackage{xcolor}
\usepackage{algpseudocode}
\usepackage{hyperref}
\usepackage{cleveref}
\usepackage{import}
\newtheorem{theorem}{Theorem}[section]
\newtheorem{lemma}[theorem]{Lemma}
\newtheorem{proposition}[theorem]{Proposition}
\newtheorem{definition}{Definition}[section]
\newtheorem{assumption}{Assumption}
\newtheorem{remark}{Remark}[section]
\newcommand{\R}{\mathbb{R}}
\newcommand{\N}{\mathbb{N}}
\DeclareMathOperator{\ri}{ri}
\DeclareMathOperator{\cl}{cl}
\DeclareMathOperator{\aff}{aff}
\DeclareMathOperator{\conv}{conv}
\newcommand{\norm}[1]{\left\|#1\right\|}
\newcommand{\inner}[2]{\langle #1,\, #2 \rangle}
\newcommand{\partialC}{\partial^C}
\newcommand{\partialL}{\partial^L}
\algnewcommand\algorithmicComment[1]{\textit{// #1}}
\algnewcommand{\COMMENT}[1]{\State \algorithmicComment{#1}}
\title{Achieving Directional-Stationarity from a Single Random Direction Step}
\author{Dan Greenstein \qquad Nadav Hallak\\[4pt]
Faculty of Data and Decision Sciences, Technion\\
\texttt{\{sdngreen, ndvhllk\}@technion.ac.il}}
\date{}
\begin{document}
\maketitle
\begin{abstract}
This paper addresses the challenge of obtaining strong optimality guarantees in constrained nonsmooth nonconvex optimization under mild regularity conditions, namely local Lipschitz continuity and existence of directional derivatives. 
While standard methods typically ensure weak stationarity notions, achieving directional (d-)stationarity remains nontrivial.
We show that a random direction exploration step is sufficient to attain d-stationarity. 
The proposed approach augments any feasible base method with one additional randomized exploration step per iteration, which samples a direction and step size and accepts the candidate based on a function value comparison.
The resulting scheme guarantees that all accumulation points are d-stationary almost surely, independently of the behavior of the underlying method. 
Moreover, under conditions on the base method, it preserves convergence rates of the base method, as established for the difference-of-convex algorithm (DCA), prox-linear-type schemes, and interpolated normalized gradient descent (INGD).
The theoretical results are complemented by numerical experiments illustrating the effect and guarantees of the exploration step.
\end{abstract}

\section{Introduction}\label{sec:intro}
\subsection{Problem formulation}
We consider the constrained nonsmooth nonconvex optimization problem
\begin{equation}\label{eq:main}
  \min_{x \in C} h(x),
\end{equation}
where $h \colon \mathbb{R}^n \to \mathbb{R}$ is locally Lipschitz and directionally differentiable, and $C \subseteq \mathbb{R}^n$ is a nonempty closed convex set. 
This setting covers a broad class of problems, including smooth-minus-convex objectives $h = f - g$ with $f \in C^1$ and $g$ convex, as well as the difference-of-convex (DC) class where both $f$ and $g$ are convex. 
Importantly, our approach does not require access to such a decomposition and applies directly to the objective $h$.

\medskip

A central challenge in nonsmooth nonconvex optimization is to obtain strong and meaningful optimality guarantees. In DC programming and related settings, several stationarity notions are available, with varying levels of strength (see, e.g., \citep{Pang2017}). 
Among these, \emph{directional (d-)stationarity} characterizes points with no feasible descent directions, defined through directional derivatives (Definition~\ref{def:d-stationary} and \Cref{ass:A2}).

Directional stationarity is stronger than the criticality notion standard in DC algorithms \citep{deOliveira2020} and than Clarke stationarity \citep{clarke1990optimization}; derivative-free schemes such as MADS~\citep{Audet2006} typically target the latter. See Remark~\ref{rem:stationarity-comparison} for limiting (Mordukhovich) stationarity and for a fuller discussion.
It is rarely guaranteed in general constrained nonsmooth nonconvex settings.
Existing methods achieving d-stationarity either rely on explicit DC decompositions and access to smooth component functions \citep{Pang2017, FengYuan2026}, or are restricted to linear feasibility sets \citep{Beck2020, Beck2022}, and their guarantees are tightly coupled to the algorithmic dynamics.

In contrast, we propose a minimal exploration step that is independent of the base algorithm and enforces d-stationarity of accumulation points. 
Hence, any feasible method can be made asymptotically d-stationary without modifying the oracle map itself; the augmentation modifies only the accepted iterate.

\paragraph{Main contributions.}

\begin{itemize}
\item \textbf{Exploration-augmented optimization.}
We introduce a random exploration mechanism that augments any feasible optimization method with one additional randomized exploration step per iteration (Algorithms~\ref{alg:ERP} and~\ref{alg:random}, Section~\ref{sec:meta}).

\item \textbf{D-stationarity via random exploration.}
We prove that, almost surely, all accumulation points are d-stationary, independently of the underlying algorithm (Theorem~\ref{thm:random}, Section~\ref{sec:convergence}). 
The result 
does not use any structural properties of the underlying method.

\item \textbf{Rate preservation and oracle compatibility.}
We give sufficient conditions under which the augmentation preserves the rate guarantee of the base method (Theorems~\ref{thm:rate-transfer} and~\ref{thm:certificate}, Section~\ref{sec:rate-transfer}; for their scope, Appendix~\ref{subsec:rate-transfer-scope}), and verify them for the difference-of-convex algorithm (DCA, Theorem~\ref{thm:dca-rate}), prox-linear-type schemes (Theorem~\ref{thm:pl-rate}), and gradient sampling (Theorem~\ref{thm:ingd}).
\end{itemize}

\subsection{Mathematical preliminaries}
We formulate two assumptions on \eqref{eq:main} and one supporting lemma that ensure well-defined directional derivatives and the continuity required for the probabilistic argument.
Throughout, $\norm{\cdot}$ denotes the Euclidean norm on $\R^n$ and $S^{n-1}$ the Euclidean unit sphere; whether the analysis of Section~\ref{sec:convergence} extends to other norms is left open.
\begin{assumption}[A1]\label{ass:A1}
$h$ is locally Lipschitz on $\R^n$ and bounded below on $C$.
\end{assumption}
\begin{assumption}[A2]\label{ass:A2}
The directional derivative $h'(x;\, d) = \lim_{t \downarrow 0}(h(x + td) - h(x))/t$ exists for every $x \in C$ and $d \in \R^n$.
\end{assumption}
\begin{lemma}[Proof in Appendix~\ref{sec:appendix Proofs of}]
    \label{ass:A3}
Let $x \in C$, and let $L$ be a Lipschitz constant of $h$ on a neighborhood of $x$. Then $|h'(x;\,d_1) - h'(x;\,d_2)| \leq L\norm{d_1 - d_2}$ for all $d_1, d_2 \in \R^n$.
\end{lemma}
These properties hold for a broad class of objectives such as $h = f - g$ with $f \in C^1$ and $g$ convex and finite-valued, as well as in the DC case where both $f$ and $g$ are convex. 
In both settings, the directional derivative admits the representation $h'(x;\, d) = f'(x;\, d) - g'(x;\, d)$, and continuity in $d$ follows from continuity/convexity respectively \citep[Theorems~23.1 and~10.1]{Rockafellar1997}.
\begin{definition}[Affine hull subspace]\label{def:affine-hull}
    Given $x^0 \in C$, the linear subspace parallel to $\aff(C)$ is defined by $L := \aff(C - x^0)$; it does not depend on the choice of $x^0 \in C$.
\end{definition}

 For $x \in C$ and $v \in \R^n$, denote the maximal feasible step length at $x$ in direction $v$ by
\[
  \tau(x;\, v) := \sup\!\bigl\{t \geq 0 : x + tv \in C\bigr\}.
\]

\begin{definition}[Cone of feasible directions]\label{def:feasible-cone}
Let $C \subseteq \R^n$ be convex and $\bar{x} \in C$. The \emph{cone of feasible directions} is $D_C(\bar{x}) := \{w \mid \exists\, \lambda > 0,\;\bar{x} + \lambda w \in C\}$; it is a convex cone.
\end{definition}

\begin{definition}[Directional (d-)stationarity]\label{def:d-stationary}
A point $\bar{x} \in C$ is \textit{d-stationary} for \eqref{eq:main} if
\begin{equation*}
  h'(\bar{x};\, d) \geq 0 \qquad \forall\, d \in D_C(\bar{x}).
\end{equation*}
\end{definition}
By Lemma~\ref{ass:A3}, the same condition is equivalent to $h'(\bar{x};\, d) \geq 0$ for every $d \in \cl(D_C(\bar{x}))$; see Lemma~\ref{lem:continuity}.

\begin{remark}[Comparison with other stationarity notions]\label{rem:stationarity-comparison}
For locally Lipschitz $h$ and closed convex $C$, \emph{Clarke stationarity} at $\bar{x} \in C$ is 
$0 \in \partialC h(\bar{x}) + N_C(\bar{x})$ \citep{clarke1990optimization}. 
Writing $h^\circ$ for the Clarke directional derivative, one has $h'(\bar{x}; d) \leq h^\circ(\bar{x}; d)$ for all $d$, with equality if $h$ is Clarke regular, hence Clarke stationarity does not in general imply d-stationarity. 
For convex $C$, d-stationarity is equivalent to $h'(\bar{x}; d) \geq 0$ for all $d \in T_C(\bar{x})$.
\emph{Mordukhovich (limiting) stationarity} ($0 \in \partialL h(\bar{x}) + N_C(\bar{x})$ \citep{RockafellarWets2009}) implies Clarke stationarity but remains weaker than d-stationarity in general.
For DC objectives $h=f-g$ with $f,g$ convex, \emph{criticality} is again weaker than d-stationarity \citep{deOliveira2020}.
\end{remark}

\section{Related Work}\label{sec:related}

We summarize the closest strands of the literature in \Cref{tab:comparison}.
To the best of our knowledge, no prior work achieves d-stationarity via random feasible-direction exploration under general convex constraints with oracle-independent guarantees.

\paragraph{DC programming.}
The standard optimality notion in DC algorithms is criticality, which is strictly weaker than d-stationarity \citep{deOliveira2020}. \citet{Pang2017} initiated a line of work establishing d-stationarity for structured DC problems in which the concave part admits a finite-max representation, by solving one convex subproblem per active component. For this finite-max structure, d-stationarity is equivalent to optimality in one convex program per active component, whereas in our setting the condition of Definition~\ref{def:d-stationary} ranges over a continuum of directions. This framework has been extended along several directions by \citet{AhnPangXin2017}, 
\citet{TaoLi2023}, 
\citet{SunWu2024}, 
and \citet{FengYuan2026}, 
all of which rely critically on the finite-max structure of the DC decomposition, explicit access to the DC decomposition and its components, and first-order information, and are inherently algorithm-specific rather than oracle-agnostic. 
\citet{Pang2018} extend this approach to multiblock objectives that need not be DC. Their guarantee is blockwise directional stationarity, and their augmented Lagrangian method meets the coupled linear constraints only in the limit.

A complementary line of work \citep{Beck2020, Beck2022} establishes d-stationarity for smooth-minus-convex objectives over polyhedral sets using feasible-direction methods, including randomized variants based on fixed finite spanning sets. In contrast, bundle methods~\citep{Joki2018, deOliveira2021} typically guarantee only criticality or Clarke stationarity.
Stochastic DC methods, for example the online sampling scheme of \citet{Ye2025}, likewise target criticality.

Finally, while the convergence rates of DCA are well understood in the unconstrained setting \citep{Abbaszadehpeivasti2024}, these results concern criticality measures and do not address d-stationarity.

\paragraph{Direct search: GSS, pattern search, and MADS.}
Generating set search~\citep{Kolda2006} and pattern search~\citep{Lewis2000} poll finitely many directions and yield KKT-type limits in the smooth case and Clarke stationary limits for Lipschitz objectives \citep[Corollary to Proposition~2.4.3]{clarke1990optimization}. MADS~\citep{Audet2006} randomizes to asymptotically dense polling (Clarke stationarity); although the direction coverage is conceptually related to our argument, MADS is a standalone mesh-based method with sufficient decrease, whereas we augment an arbitrary oracle and target d-stationarity under general convex constraints. The feasible-direction methods of \citet{Beck2020, Beck2022} reach d-stationarity only under smooth-minus-convex structure over polyhedral sets.


\paragraph{Randomized search outside direct search.}
\citet{Gratton2015} analyze probabilistic-descent direct search, yielding almost-sure convergence and $O(1/\sqrt{k})$ gradient-norm rates for smooth objectives; \citet{Gratton2019} treat linear constraints analogously. These schemes remain direct-search templates targeting gradient- or Clarke-type stationarity. Classical random search \citep{SolisWets1981} and sufficient-descent schemes \citep{Wardi1989} give value- or stationarity-type guarantees under coverage or internal descent, but are standalone algorithms. Our augmentation instead layers one random feasible-direction probe on an arbitrary oracle and targets d-stationarity of accumulation points under convex constraints.



\begin{table}[t]
\centering
\caption{Comparison of related fields. 
}
\label{tab:comparison}
\smallskip
\begin{tabularx}{\linewidth}{@{}>{\raggedright\arraybackslash}X>{\raggedright\arraybackslash}X>{\raggedright\arraybackslash}X>{\raggedright\arraybackslash}X@{}}
\toprule
Field / Work & Target & Mechanism & Nonsmooth DC? \\
\midrule
DC methods \citep{Pang2017} & D-stationarity & Convex subproblem & $g = \max_i \psi_i$, $\psi_i \in C^1$ only \\
GSS / GPS \citep{Kolda2006} & Clarke stat. & Finite spanning set & No \\
MADS \citep{Audet2006} & Clarke stat. & Dense random poll & No \\
Prob.\ direct search \citep{Gratton2015} & Grad.\ stat. & Random sphere dir. & No \\
Random search \citep{SolisWets1981, Wardi1989} & Value / stationarity & Coverage / sufficient descent & No \\
\midrule
This work & D-stationarity & Random sphere dir. & Yes \\
\bottomrule
\end{tabularx}
\end{table}

\section{The Random Exploration Procedure}\label{sec:meta}
The Random Exploration Procedure (REP) augments a base optimization method with a simple randomized exploration step. 
At the current iterate, REP samples a direction in the affine hull and a candidate step size, accepting the move only if it yields a feasible sufficient decrease in a regularized objective; it does not rely on gradients or problem structure and uses only function value comparisons.
This regularization guarantees sufficient-decrease-type behavior.


\begin{algorithm}
\caption{Random Exploration Procedure (REP)}
\label{alg:ERP}
\begin{algorithmic}[1]
\Require $x \in C$; $\gamma > 0$, $r \in (0,\infty)$.
  \State Sample $v_{ex} \sim \mathcal{D}_v$ with  $\mathrm{supp}(\mathcal{D}_v) = S^{n-1} \cap L$
  \State Sample $\hat{t} \sim \mathcal{D}_t$ with  $\mathrm{supp}(\mathcal{D}_t) = [0,r]$
  \State $t_{ex} \leftarrow 
    \begin{cases}
    \hat{t} & \text{if } x + \hat{t} v_{ex} \in C \text{ and } h(x + \hat{t} v_{ex}) + \frac{\gamma}{2}(\hat{t})^2 < h(x), \\
    0 & \text{otherwise.}
    \end{cases}$
    \State \textbf{Return}  $y \leftarrow x + t_{ex} v_{ex}$
\end{algorithmic}
\end{algorithm}

\begin{remark}[Sampling distributions]\label{rem:sampling}
The direction distribution $\mathcal{D}_v$ in Algorithm~\ref{alg:ERP} is supported on $S^{n-1} \cap L$ and assigns strictly positive probability to every measurable subset of $S^{n-1} \cap L$ with strictly positive surface measure. The step-size distribution $\mathcal{D}_t$ is supported on $[0,r]$ and assigns positive probability to every interval $[a,b] \subseteq (0,r]$.
Let $\mathcal{F}_{k-1}$ denote the $\sigma$-algebra generated by $\{x^0, v^0_{ex}, \hat{t}^0, \ldots, x^{k-1}, v^{k-1}_{ex}, \hat{t}^{k-1}, x^k\}$. We assume that the pairs $(v_{\mathrm{ex}}^k, \hat t^k)$ are identically distributed, and that $v_{\mathrm{ex}}^k$ and $\hat t^k$ are independent of each other and of $\mathcal{F}_{k-1}$.
\end{remark}

\Cref{alg:random} augments a feasible oracle with a randomized exploration step: the oracle proposes a candidate, while exploration samples a direction and step size to generate a competing feasible point accepted by value comparison. Almost surely, every neighborhood of relative-interior feasible direction at any accumulation point is sampled infinitely often, yielding d-stationarity via Borel--Cantelli.


\begin{algorithm}
\caption{Random exploration procedure implementation}
\label{alg:random}
\begin{algorithmic}[1]
\Require $x^0 \in C$; oracle $\mathcal{O}_C$ over $C$; $\gamma > 0$, $r \in (0,\infty)$.
\For{$k = 0, 1, 2, \ldots$}
  \State Set $y^{k+1} \leftarrow \mathrm{REP}(x^k)$ and  $z^{k+1} \leftarrow \mathcal{O}_C(x^k)$
  \State $x^{k+1} \leftarrow \arg\min_{u\,\in\,\{y^{k+1},\,z^{k+1}\}} h(u)$
\EndFor
\end{algorithmic}
\end{algorithm}

\begin{remark}[Implementation cost]\label{rem:cost}
\Cref{alg:random} requires one feasibility check 
and at most one function evaluation, 
which can be skipped when the candidate is infeasible, per exploration step. 
\end{remark}


\section{Convergence Analysis}\label{sec:convergence}

This section establishes that the exploration step enforces asymptotic d-stationarity, independently of the oracle; some proofs are deferred to Appendix~\ref{sec:appendix Proofs of}.


\begin{theorem}[Asymptotic d-stationarity]\label{thm:random}
Let  \Cref{ass:A1} and \Cref{ass:A2} hold, and let $\{x^k\}$ be generated by Algorithm~\ref{alg:random} with any oracle $\mathcal{O}_C$ producing feasible iterates. Then:
\begin{enumerate}
\item[(i)] $\sum_{k=0}^{\infty}(t^k_{ex})^2 < \infty$, and in particular $t^k_{ex} \to 0$.
\item[(ii)] Almost surely, every accumulation point of $\{x^k\}$ is d-stationary.
\end{enumerate}
\end{theorem}


\begin{remark}[Existence of accumulation points]\label{rem:accumulation}
A sufficient condition for the existence of an accumulation point is boundedness of the level set, $\{x \in C : h(x) \le h(x^0)\}.$
Indeed, greedy selection ensures $h(x^{k+1}) \le h(x^k)$ for all $k$, hence $h(x^k) \le h(x^0)$ and $\{x^k\}$ remains in this set. If it is bounded, then $\{x^k\}$ is bounded in $\R^n$ and admits an accumulation point by Bolzano--Weierstrass.
\end{remark}


The proof of Theorem~\ref{thm:random} views the exploration step as a probe that is tried at every iteration and moves only when it finds sufficient decrease. Such moves are self-limiting: each accepted step lowers $h$ by at least $\frac{\gamma}{2}$ times its squared length, and $h$ is bounded below on $C$, so the accepted steps must become negligible, which is part~(i). A feasible descent direction at an accumulation point $x^*$ would contradict this. Near $x^*$, a decrease of order $t$ dominates the quadratic acceptance penalty $\frac{\gamma}{2}t^2$, yielding a window of steps, of lengths bounded away from zero, that remain feasible and pass the acceptance test from every iterate sufficiently close to $x^*$ (Figure~\ref{fig:rep-mechanism}, Appendix~\ref{sec:appendix Proofs of}). The iterates come arbitrarily close to $x^*$ infinitely often and the sampling distributions have full support, so the exploration step almost surely falls into this window infinitely often and keeps making steps that do not shrink. Lemma~\ref{lem:step-bound} supplies the feasibility radius of the window, Lemma~\ref{lem:BC-directions} the repeated sampling, and the decrease estimates in the proof the acceptance.

Throughout this section, let $\mathcal{Q}\subseteq S^{n-1}\cap L$ be a fixed countable dense set.

\begin{lemma}[Joint density of exploration direction, proximity, and step size]\label{lem:BC-directions}
There exists an event $\Omega_{BC}$ with $\mathbb{P}(\Omega_{BC})=1$ such that, on $\Omega_{BC}$, for every $0 < a < b \leq r$, every $\rho > 0$, every accumulation point $x^*$ of $\{x^k\}$, and every $v \in \mathcal{Q}$, there exists a subsequence $J \subseteq \N$ such that, as $j \to \infty$ along $J$, it holds that: $x^j \to x^*,\ v^j_{ex} \to v,\ \norm{v^j_{ex} - v} < \rho,\ \hat{t}^j \in [a,\, b].$
\end{lemma}

\begin{lemma}[Uniform feasibility radius]\label{lem:step-bound}
Let $x^* \in C$ and $v \in \ri(D_C(x^*)) \cap L$. There exist $\rho > 0$, $\bar{\varepsilon} > 0$, and $\delta > 0$ such that: (i)  $B(v,\, \rho) \cap L \subseteq \ri(D_C(x^*))$; (ii) $\tau(x;\, w) \geq \bar{\varepsilon}$ for every $x \in C$ with $\norm{x - x^*} < \delta$ and every $w \in B(v,\, \rho) \cap L$.
\end{lemma}

\begin{lemma}[Continuity and extension]\label{lem:continuity}
The directional derivative map $d \mapsto h'(x^*;\, d)$ is continuous and positively homogeneous of degree one. Moreover, if $h'(x^*;\, v) \geq 0$ for all $v \in \ri(D_C(x^*))$, then $h'(x^*;\, d) \geq 0$ for all $d \in \cl(D_C(x^*))$.
\end{lemma}




\begin{proof}[Proof of \Cref{thm:random}]
\textit{(i) Step-size decay.} Since $t = 0$ is always available in the two-point comparison and $t^k_{ex}$ minimizes over $\{0, \hat{t}^k\}$:
  $h\!\left(x^k + t^k_{ex}\,v^k_{ex}\right) + \tfrac{\gamma}{2}(t^k_{ex})^2 \leq h(x^k).$
  
By greedy selection, $h(x^{k+1}) \leq h(y^{k+1})$, so $h(x^k) - h(x^{k+1}) \geq \frac{\gamma}{2}(t^k_{ex})^2$. Summing gives $\sum_k (t^k_{ex})^2 \leq \frac{2}{\gamma}(h(x^0) - \inf_C h) < \infty$, hence $t^k_{ex} \to 0$.

\textit{(ii) D-stationarity.} Let $\Omega_{BC}$ be the event from Lemma~\ref{lem:BC-directions}, which guarantees simultaneous subsequences for all directions in $\mathcal{Q}$. Fix $\omega \in \Omega_{BC}$ and work pathwise. Let $x^*$ be an accumulation point of $\{x^k(\omega)\}$, and enumerate $\mathcal{Q}\cap \ri(D_C(x^*))$ as $\{v_i\}_{i=0}^{\infty}$. Since $\ri(D_C(x^*))\cap S^{n-1}$ is relatively open in $S^{n-1}\cap L$ and $\mathcal{Q}$ is dense in $S^{n-1}\cap L$, the set $\{v_i\}$ is dense in $\ri(D_C(x^*))\cap S^{n-1}$.

Fix $i \in \N$ and suppose for contradiction $h'(x^*;\, v_i) = -\varepsilon < 0$ for some $\varepsilon > 0$.

By Lemma~\ref{lem:step-bound} applied with $v = v_i$, there exist $\rho_i > 0$, $\bar{\varepsilon}_i > 0$, and $\delta_i > 0$ such that $B(v_i,\, \rho_i) \cap L \subseteq \ri(D_C(x^*))$, and $\tau(x;\, w) \geq \bar{\varepsilon}_i$ for every $x \in C$ with $\norm{x - x^*} < \delta_i$ and every $w \in B(v_i,\, \rho_i) \cap L$.

By local Lipschitz continuity of $h$ at $x^*$, there exist $R_h>0$ and $L_h>0$ such that $h$ is $L_h$-Lipschitz on $B(x^*,R_h)$.

Next choose the comparison scale. By definition of the directional derivative with $h'(x^*;\, v_i) = -\varepsilon$, there exists $\tau_0 > 0$ such that $h(x^* + t v_i) - h(x^*) \leq -\tfrac{\varepsilon}{2}\,t$ for every $t \in (0,\, \tau_0]$. Set
\[
  t^* := \min\!\left\{\bar{\varepsilon}_i,\;\tau_0,\;\frac{\varepsilon}{4\gamma},\;r,\;\frac{R_h}{4}\right\} > 0.
\]

Now lock direction and step-size sampling simultaneously. Since $\omega \in \Omega_{BC}$ and the parameters are now fixed, Lemma~\ref{lem:BC-directions} gives an infinite subsequence $\mathcal{J}_i \subseteq \N$ along which $x^j \to x^*,  v^j_{ex} \to v_i,  \norm{v^j_{ex} - v_i} < \rho_i,  \hat{t}^j \in [t^*/2,\, t^*].$

For $j \in \mathcal{J}_i$ sufficiently large, $\norm{x^j - x^*} < \delta_i$; the feasibility bound above then gives $\tau(x^j;\, v^j_{ex}) \geq \bar{\varepsilon}_i \geq t^*$, so $x^j + t v^j_{ex} \in C$ for every $t \in [0,\, t^*]$.

The key point is that the same strict decrease estimate holds uniformly for every $t \in [t^*/2,\, t^*]$. 
Since $x^j \to x^*$ along $\mathcal{J}_i$, for all sufficiently large $j \in \mathcal{J}_i$ we have $\norm{x^j-x^*}<R_h/4$. For such $j$ and any $t\in[0,t^*]$, $\norm{x^*+t v_i-x^*}=t\le t^*\le R_h/4<R_h,$
and, using $\norm{v_{ex}^j}=1$,
\[
\norm{x^j+t v_{ex}^j-x^*}\le \norm{x^j-x^*}+t\,\norm{v_{ex}^j}
<R_h/4+t^*\le R_h/2<R_h.
\]
Hence both points $x^j+t v_{ex}^j$ and $x^*+t v_i$ lie in $B(x^*,R_h)$, so the $L_h$-Lipschitz bound is valid. Therefore, for any $t \in [t^*/2,\, t^*]$ and all sufficiently large $j \in \mathcal{J}_i$, decompose
\begin{align*}
  h(x^j + t v^j_{ex}) - h(x^j) + \tfrac{\gamma}{2}t^2
  &= \bigl[h(x^j + t v^j_{ex}) - h(x^* + t v_i)\bigr]
   + \bigl[h(x^* + t v_i) - h(x^*)\bigr] \\
  &\quad + \bigl[h(x^*) - h(x^j)\bigr] + \tfrac{\gamma}{2}t^2 \\
  &\leq L_h\bigl(\norm{x^j - x^*} + t\,\norm{v^j_{ex} - v_i}\bigr)
     - \tfrac{\varepsilon}{2}t
     + L_h\norm{x^j - x^*}
     + \tfrac{\gamma}{2}t^2.
\end{align*}
Since $t \in [t^*/2,\, t^*]$, $\tfrac{\varepsilon}{2}t \geq \tfrac{\varepsilon t^*}{4}$; and $t^* \leq \varepsilon/(4\gamma)$ gives $\tfrac{\gamma}{2}t^2 \leq \tfrac{\gamma}{2}(t^*)^2 \leq \tfrac{\varepsilon t^*}{8}$. 
Since $x^j \to x^*$ and $v^j_{ex} \to v_i$ along $\mathcal{J}_i$, and since $[t^*/2,\, t^*]$ is compact and $x^j \to x^*$, $v^j_{ex} \to v_i$, we may pass to a tail subsequence (still denoted $\mathcal{J}_i$) on which, uniformly in $t \in [t^*/2,\, t^*]$,
\[
  L_h\bigl(\norm{x^j - x^*} + t\,\norm{v^j_{ex} - v_i}\bigr) + L_h\norm{x^j - x^*}
  \;\leq\; 2 L_h\norm{x^j - x^*} + L_h\, t^*\norm{v^j_{ex} - v_i}
  \;\leq\; \tfrac{\varepsilon t^*}{16}.
\]
Therefore, for every $t \in [t^*/2,\, t^*]$ we have that
\[
  h(x^j + t v^j_{ex}) - h(x^j) + \tfrac{\gamma}{2}t^2
  \;\leq\; \tfrac{\varepsilon t^*}{16} - \tfrac{\varepsilon t^*}{4} + \tfrac{\varepsilon t^*}{8}
  \;=\; -\tfrac{\varepsilon t^*}{16} \;<\; 0.
\]
Applied at $t = \hat{t}^j \in [t^*/2,\, t^*]$, together with the feasibility established above, this strict negativity forces the two-point comparison to select $t^j_{ex} = \hat{t}^j \geq t^*/2$ for every $j \in \mathcal{J}_i$.

This yields the contradiction: for infinitely many $j \in \mathcal{J}_i$, $t^j_{ex} \geq t^*/2$, so $\sum_k (t^k_{ex})^2 \geq \sum_{j \in \mathcal{J}_i}(t^*/2)^2 = +\infty$, contradicting part~(i). Hence $h'(x^*;\, v_i) \geq 0$.

Since $\{v_i\}$ is dense in $\ri(D_C(x^*)) \cap S^{n-1}$, Lemma~\ref{lem:continuity} (continuity and positive homogeneity of $h'(x^*;\, \cdot)$, together with the relative-interior-to-closure extension) yields $h'(x^*;\, d) \geq 0$ for every $d \in \cl(D_C(x^*))$. Hence $x^*$ is d-stationary. As $x^*$ was an arbitrary accumulation point on this fixed $\omega \in \Omega_{BC}$, every accumulation point is d-stationary on $\Omega_{BC}$, proving the almost-sure claim.
\end{proof}

The preceding argument requires only that the oracle outputs feasible points $z^{k+1} \in C$; beyond this requirement, $\mathcal{O}_C$ may be arbitrary, while d-stationarity is enforced solely by the exploration step. In particular, $\mathcal{O}_C$ may be randomized or adaptive, provided each exploration sample $(v^k_{ex}, \hat{t}^k)$ is independent of the $\sigma$-algebra $\mathcal{F}_{k-1}$, as assumed in Remark~\ref{rem:sampling}.

\section{Oracle Instantiations}\label{sec:oracles}

For DCA, prox-linear methods, and gradient sampling (INGD), the oracle supplies a stationarity measure and a rate, while Theorem~\ref{thm:random} upgrades accumulation points to d-stationarity almost surely.
Under the additional assumptions that $C = \R^n$, $h$ is globally Lipschitz, and the exploration step draws its direction and step size independently and uniformly, exploration alone also yields an expected iteration bound, exponential in $n$, for Goldstein stationarity (Appendix~\ref{sec:appendix-rep-finite}).
Throughout this section, $h^* := \inf_{x \in C} h(x)$, which is finite by Assumption~A1, and $\Delta_0 := h(x^0) - h^*$, where $x^0$ is the initial point of Algorithm~\ref{alg:random}.


\subsection{DCA for constrained DC programming}\label{sec:dc}
Consider the constrained difference-of-convex problem
\begin{equation}\label{eq:dca}
    \min_{x \in C} h(x) := f_1(x) - f_2(x),
\end{equation}
where $C \subseteq \R^n$ is closed and convex, $f_1, f_2 \colon \R^n \to \R$ are finite-valued, $f_1$ is $\mu_1$-strongly convex and $f_2$ is $\mu_2$-strongly convex with $\mu_1 + \mu_2 > 0$, and $h$ is bounded below on $C$.
We assume that the DCA subproblem below has a solution at every point of $C$, as it does when $\mu_1 > 0$.

Given $x^k \in C$, the DCA oracle produces
\[
z^{k+1} \in \arg\min_{x \in C} \; f_1(x) - \langle g_2^k, x \rangle, 
\quad g_2^k \in \partial f_2(x^k).
\]



The augmented method applies Algorithm~\ref{alg:random} with this oracle. 
The key observation is that the rate analysis of DCA relies only on the descent inequality
$h(x^{k+1}) \leq h(z^{k+1}),$
which is preserved by the greedy selection step regardless of whether the exploration step is accepted; the proof is deferred to Appendix~\ref{sec:appendix-dca}.



\begin{theorem}[Rate preservation under DCA]\label{thm:dca-rate}
Under the assumptions above, after $N$ iterations of Algorithm~\ref{alg:random} with the DCA oracle, $\min_{0 \leq k \leq N-1} \norm{x^k - z^{k+1}}
\leq 
\sqrt{\frac{2\,(h(x^0) - h^*)}{(\mu_1 + \mu_2)\,N}}.$
\end{theorem}



Thus, the augmentation preserves the $O(1/\sqrt{N})$ rate on the natural oracle residual (Remark~\ref{rem:dca-optimality-measure}), while Theorem~\ref{thm:random} ensures that every accumulation point is almost surely d-stationary, which is strictly stronger than the criticality guarantee of DCA.

\subsection{General rate transfer}\label{sec:rate-transfer}

The proof of Theorem~\ref{thm:dca-rate} uses the DCA oracle only through Lemma~\ref{lem:dca-descent} (a lower bound on $h(x^k) - h(z^{k+1})$), and the exploration step only through greedy selection, which gives, whether or not the exploration step is accepted,
\begin{equation*}
  h(x^{k+1}) \leq h(z^{k+1}) \qquad \text{for every } k. \tag{G}\label{eq:greedy-descent}
\end{equation*}
As in Algorithm~\ref{alg:random}, we write $z^{k+1} \leftarrow \mathcal{O}_C(x^k)$: the oracle returns a single point. When its defining rule admits several outputs, $\mathcal{O}_C(x)$ denotes the one it returns; when it is randomized, $\mathcal{O}_C(x)$ is a random variable and $z^{k+1}$ is its realization at iteration $k$. We consider three properties of $\mathcal{O}_C$.
\begin{itemize}
  \item[(O1)] Restartable: $\mathcal{O}_C$ is defined on $C$, and $\mathcal{O}_C(x)$ depends only on $x$, not on the iteration index or on state carried across calls. For a randomized oracle, $\mathcal{O}_C(x) = \Psi(x, \xi)$ for a fixed measurable map $\Psi$, with $\xi$ drawn at each call from the same distribution, independently of the past.
  \item[(O2)] Feasible: $\mathcal{O}_C(x) \in C$ for every $x \in C$, as Algorithm~\ref{alg:random} requires of every oracle.
  \item[(O3)] Non-ascending: $h(\mathcal{O}_C(x)) \leq h(x)$ for every $x \in C$.
\end{itemize}
For a randomized oracle, (O2), (O3) and \eqref{eq:calibration} below are required for every realization of $\xi$.

\begin{proposition}[Budget]\label{prop:budget}
Consider a run of Algorithm~\ref{alg:random} along which $h(z^{k+1}) \leq h(x^k)$ for every $k$, as holds under (O3). Then, along this run, for every $K \geq 1$ and every $\eta > 0$,
\begin{equation*}
\begin{aligned}
  \sum_{k=0}^{K-1} \big(h(x^k) - h(z^{k+1})\big) &\leq \Delta_0, \\
  \min_{k<K} \big(h(x^k) - h(z^{k+1})\big) &\leq \frac{\Delta_0}{K}, \qquad
  \#\big\{k<K : h(x^k) - h(z^{k+1}) \geq \eta\big\} &\leq \frac{\Delta_0}{\eta}.
\end{aligned}
\end{equation*}
\end{proposition}

The proof (Appendix~\ref{sec:appendix-rate-transfer}) telescopes \eqref{eq:greedy-descent}. The exploration randomness does not enter it, so the bounds hold surely, and they hold verbatim for the base method without exploration, for which $z^{k+1} = x^{k+1}$.

Let $\Lambda(x) := h(x) - h(\mathcal{O}_C(x))$ denote the oracle decrement at $x$, a function on $C$ for a deterministic oracle and a random variable for a randomized one; its value at $x^k$ is the summand in Proposition~\ref{prop:budget}. A rate is stated in an optimality measure $\sigma \colon C \times C \to \R_{\geq 0}$ of the base method. Under (O1) and (O3), we say that $\mathcal{O}_C$ admits a monotone calibration for $\sigma$ if, for some nondecreasing $\varphi \colon \R_{\geq 0} \to \R_{\geq 0}$,
\begin{equation*}
  \sigma\big(x, \mathcal{O}_C(x)\big) \leq \varphi\big(\Lambda(x)\big) \qquad \text{for every } x \in C. \tag{$\star$}\label{eq:calibration}
\end{equation*}

\begin{theorem}[Rate transfer]\label{thm:rate-transfer}
Let $\sigma \colon C \times C \to \R_{\geq 0}$, and let $\varphi \colon \R_{\geq 0} \to \R_{\geq 0}$ be nondecreasing. Consider a run of Algorithm~\ref{alg:random} along which, for every $k \geq 0$,
\begin{equation*}
  h(z^{k+1}) \leq h(x^k) \quad \text{and} \quad \sigma(x^k, z^{k+1}) \leq \varphi\big(h(x^k) - h(z^{k+1})\big). \tag{T}\label{eq:trajectory}
\end{equation*}
Then, along this run, for every $K \geq 1$ and every $\eta > 0$,
\begin{equation*}
  \min_{k<K} \sigma(x^k, z^{k+1}) \leq \varphi\Big(\frac{\Delta_0}{K}\Big), \qquad
  \#\big\{k<K : \sigma(x^k, z^{k+1}) > \varphi(\eta)\big\} \leq \frac{\Delta_0}{\eta}.
\end{equation*}
\end{theorem}

If $\mathcal{O}_C$ satisfies (O1) and (O3) and admits a monotone calibration for $\sigma$, then \eqref{eq:trajectory} holds along every run, since $x^k \in C$ and $z^{k+1}$ is the output of $\mathcal{O}_C$ at $x^k$. This is the checkable sufficient condition for \eqref{eq:trajectory}: a descent lemma proved at an arbitrary point of $C$ yields \eqref{eq:calibration}, and therefore applies at iterates produced by exploration, whereas for an oracle that carries state across calls it holds only along the base method's own trajectory. For the DCA oracle of Section~\ref{sec:dc}, with choices of $g_2 \in \partial f_2(x)$ and of the minimizer that depend only on $x$, (O1)--(O3) hold, and Lemma~\ref{lem:dca-descent}, whose proof applies at every $x \in C$, is \eqref{eq:calibration} with $\sigma(x,z) = \norm{x - z}$ and $\varphi(d) = \sqrt{2d/(\mu_1+\mu_2)}$; Theorem~\ref{thm:rate-transfer} then gives Theorem~\ref{thm:dca-rate} with the same constant. If $\varphi$ is continuous at $0$ with $\varphi(0) = 0$, as here, the first bound of Theorem~\ref{thm:rate-transfer} is a convergence rate whose order $\varphi$ determines; if $\varphi$ is a step function, the second bounds the number of iterations at which $\sigma$ exceeds a threshold.

Oracles that certify stationarity rather than measure it admit no calibration with $\varphi(0) = 0$ for a measure $\sigma$ with $\sigma(x, z) > 0$ whenever $x$ is not stationary, since a call can return $z = x$, so that $\Lambda(x) = 0$, at such a point. For them, Proposition~\ref{prop:budget} yields a hitting time instead. Suppose that the call at iteration $k$ also returns a label $\ell^k \in \{\mathrm{CERT}, \mathrm{PROG}, \mathrm{FAIL}\}$. Let $\mathcal{G}_{k-1}$ be the $\sigma$-algebra generated by $x^0$ and by the exploration samples, oracle outputs and labels of iterations $0, \dots, k-1$ (so $\mathcal{G}_{-1}$ is generated by $x^0$), so that $x^k$ is $\mathcal{G}_{k-1}$-measurable. Consider the following properties.
\begin{itemize}
  \item[(P1)] $\ell^k = \mathrm{CERT}$ implies that $x^k$ satisfies the target stationarity criterion.
  \item[(P2)] $\ell^k = \mathrm{PROG}$ implies $h(x^k) - h(z^{k+1}) \geq \bar\eta$, for a constant $\bar\eta > 0$.
  \item[(P3)] $\mathbb{P}(\ell^k = \mathrm{FAIL} \mid \mathcal{G}_{k-1}) \leq q$ almost surely, for every $k$.
\end{itemize}

\begin{theorem}[Certificate within a finite horizon]\label{thm:certificate}
Let $\mathcal{O}_C$ satisfy (O3) and (P1)--(P3) with $q < 1/e$. There is a constant $\kappa_q$, depending only on $q$, such that for every $\zeta \in (0,1)$, with probability at least $1 - \zeta$, some call $\mathcal{O}_C(x^k)$ with $k < \kappa_q\,(1 + \Delta_0/\bar\eta + \log(1/\zeta))$ returns $\mathrm{CERT}$. An explicit horizon is given in Appendix~\ref{sec:appendix-rate-transfer}.
\end{theorem}

The proof is in Appendix~\ref{sec:appendix-rate-transfer}. Under (O1), with the label computed from the same draw, a failure bound at every fixed point of $C$ gives (P3), since the draw at iteration $k$ is independent of $\mathcal{G}_{k-1}$ and $x^k$ is $\mathcal{G}_{k-1}$-measurable.

The conditions of Theorems~\ref{thm:rate-transfer} and~\ref{thm:certificate} are sufficient for rate transfer, and we do not claim that they are necessary. Since \eqref{eq:greedy-descent} bounds $h(x^{k+1})$ but does not give $x^{k+1} = z^{k+1}$, the analysis of a base method yields \eqref{eq:trajectory} only if its descent inequality holds at the points that exploration produces. Condition (O1) fails for oracles that carry state across calls, such as accelerated, bundle and quasi-Newton methods, whose analysis does not in general supply \eqref{eq:trajectory} along the augmented run. Deriving \eqref{eq:trajectory} or (P3) from the base analysis thus relies on (O1), whereas Theorem~\ref{thm:random} allows adaptive oracles. Condition (O2) fails for penalty and augmented Lagrangian methods, whose iterates are feasible only in the limit; they lie outside Algorithm~\ref{alg:random}. Condition \eqref{eq:calibration} can fail under (O1)--(O3), for instance for the distance to a solution set. Condition (O3) fails for methods that may increase $h$, such as nonmonotone line searches. These cases, and monotone accelerated variants such as MFISTA \citep{Beck2009a}, which we leave open, are discussed in Appendix~\ref{subsec:rate-transfer-scope}.

\subsection{Prox-linear Oracle}\label{subsec:pl-oracle}
We next consider composite objectives of the form
\begin{equation}\label{eq:prox-linear}
    \min_{x \in \R^n} h(x) := c(F(x)),
\end{equation}
where $h$ is bounded below, $F$ is continuously differentiable with Lipschitz Jacobian and $c$ is globally Lipschitz 
and directionally differentiable. 
We assume in addition that $c$ is \emph{Hadamard directionally differentiable} on $\R^m$: for every $v, w \in \R^m$ the limit $ c'(v;\, w) := \lim_{(t,\,w') \to (0^+,\, w)} (c(v + t w') - c(v))/t$
exists, and for each fixed $v$ the map $w \mapsto c'(v;\, w)$ is continuous on $\R^m$.
Since $c$ is Lipschitz and $F$ is continuously differentiable, $h = c \circ F$ is locally Lipschitz, and the chain rule for Hadamard directional derivatives shows that $h$ is Hadamard directionally differentiable at every $x \in \R^n$, with
\begin{equation}\label{eq:pl-hadamard-chain}
  h'(x;\, d) = c'\!\bigl(F(x);\, \nabla F(x)\, d\bigr).
\end{equation}
Hence $h$ satisfies Assumptions~A1 and~A2.

Given $x^k$, define the local model
$m^k(y) := c\big(F(x^k) + \nabla F(x^k)(y - x^k)\big) + \frac{\rho}{2}\norm{y - x^k}^2,$
and, because $c$ need not be convex, let the oracle return $z^{k+1}$ from a first-order subproblem variant (Appendix~\ref{sec:appendix-pl}) while enforcing
\begin{equation}\label{eq:pl-model-decrease}
    m^k(z^{k+1}) \leq m^k(x^k) = h(x^k).
\end{equation}

When $\rho$ is sufficiently large, this implies $h(z^{k+1}) \leq h(x^k)$ and $h(x^k) - h(z^{k+1}) \geq \alpha \norm{z^{k+1} - x^k}^2$ for some $\alpha > 0$ (Lemma~\ref{lem:pl-descent}). Along a run, this is \eqref{eq:trajectory} with $\sigma(x,z) = \norm{x - z}$ and $\varphi(d) = \sqrt{d/\alpha}$, so Theorem~\ref{thm:rate-transfer} gives the following result (proof in Appendix~\ref{subsec:app-pl-descent-and-rate}).

\begin{theorem}[Rate preservation for prox-linear oracles]\label{thm:pl-rate}
Suppose the oracle satisfies \eqref{eq:pl-model-decrease} and that there exists $\alpha > 0$ such that $h(x^k) - h(z^{k+1}) \geq \alpha \norm{z^{k+1} - x^k}^2$ for every $k$. 
Then, after $N$ iterations of Algorithm~\ref{alg:random}, $\min_{0 \leq k \leq N-1} \norm{x^k - z^{k+1}}
\leq 
\sqrt{\frac{h(x^0) - h^*}{\alpha N}}.$
\end{theorem}

The prox-linear framework includes a broad class of nonsmooth nonconvex problems; different inner solvers may yield different notions of stationarity (e.g., Clarke or Mordukhovich stationarity).
Regardless of these distinctions, Theorem~\ref{thm:random} upgrades all accumulation points of the augmented method to d-stationarity almost surely.

\subsection{Gradient Sampling Oracle and Goldstein Stationarity}\label{subsec:gs-oracle}

Let $C = \R^n$, and let $h \in \mathcal{F}(\Delta, L)$, the class of $L$-Lipschitz, Hadamard directionally differentiable functions with $h(x^0) - \inf h \leq \Delta$ studied by \citet{Zhang2020}, who access $h$ through a generalized-gradient oracle $\mathsf{O}$ that returns $h(y)$ and an element of $\partialC h(y)$ (Appendix~\ref{sec:appendix-ingd}). Such $h$ satisfies Assumptions~A1 and~A2. Rates for this class concern a relaxation of stationarity.

\begin{definition}[Goldstein stationarity \citep{Goldstein1977}]\label{def:goldstein}
Let $\delta, \varepsilon > 0$. The Goldstein $\delta$-subdifferential of $h$ at $x$ is $\partial_\delta h(x) := \conv \bigcup_{y \in B[x,\delta]} \partialC h(y)$, and $x$ is $(\delta, \varepsilon)$-Goldstein stationary if $\mathrm{dist}(0, \partial_\delta h(x)) \leq \varepsilon$.
\end{definition}

The single-pass INGD oracle $\mathcal{O}_C(x)$ runs once, at the anchor $x$ and with fresh randomness, the inner loop of interpolated normalized gradient descent (INGD) \citep[Algorithm~1]{Zhang2020}, with $J := \lceil 48L^2/\varepsilon^2 \rceil$ steps. The loop keeps an element $m_j$ of $\partial_\delta h(x)$. It returns $x$ with the label $\mathrm{CERT}$ when $\norm{m_j} \leq \varepsilon$, returns a trial point with the label $\mathrm{PROG}$ when that point passes a sufficient-decrease test, and returns $x$ with the label $\mathrm{FAIL}$ when $J$ steps pass without either (Appendix~\ref{sec:appendix-ingd}). Since a $\mathrm{FAIL}$ call returns $z^{k+1} = x^k$ at a point that need not be stationary, this oracle admits no calibration (Section~\ref{sec:rate-transfer}), and its rate follows from Theorem~\ref{thm:certificate}.

\begin{theorem}[INGD instance; proof in Appendix~\ref{sec:appendix-ingd}]\label{thm:ingd}
Let $\delta > 0$, $0 < \varepsilon \leq L$ and $\zeta \in (0,1)$, and run Algorithm~\ref{alg:random} with the single-pass INGD oracle. This oracle satisfies (O1)--(O3) and (P1)--(P3) with $\bar\eta = \delta\varepsilon/4$ and $q = 1/4$. With probability at least $1 - \zeta$, some call $\mathcal{O}_C(x^k)$ returns $\mathrm{CERT}$, so that $x^k$ is $(\delta, \varepsilon)$-Goldstein stationary, and the calls up to and including it make at most $800\, L^2 \varepsilon^{-2} \big(1 + \Delta_0 \delta^{-1} \varepsilon^{-1} + \log(1/\zeta)\big)$ calls to $\mathsf{O}$.
\end{theorem}

Up to the constant, the term $\Delta_0 L^2 \varepsilon^{-3} \delta^{-1}$ of this bound is the leading term of the bound of \citet[Theorem~8]{Zhang2020} for INGD itself, in which $\log(1/\zeta)$ enters through the factor $\log(4\Delta/(\zeta\delta\varepsilon))$ rather than additively. Exploration makes no calls to $\mathsf{O}$ and adds at most one function evaluation per iteration (Remark~\ref{rem:cost}), against up to $J + 1$ calls to $\mathsf{O}$ per oracle call. Theorem~\ref{thm:ingd} certifies an iterate at a finite time; by Theorem~\ref{thm:random}, every accumulation point of $\{x^k\}$ is d-stationary almost surely.

\section{Numerical Experiments}\label{sec:experiments-story}

Our experiments evaluate whether one exploration step per iteration improves the terminal solution relative to the base oracle across representative benchmarks, using long runs to approach the asymptotic regime implied by the theory.
The full details are deferred to Appendix \ref{sec:appendix-experiments-v2}.


\paragraph{Protocol in brief.}
For each instance, we compare the base oracle with its augmented version over three seeds, reporting $\Delta = h_{\text{base}} - \mathrm{median}_s\, h_{\text{aug},s}$, where $h_{\text{base}}$ is the objective value at the final iterate of the base oracle, $h_{\text{aug},s}$ is that of the augmented method with seed $s$, and $\mathrm{median}_s$ is the median over the three seeds. Positive values of $\Delta$ favor augmentation. Table~\ref{tab:experiments-story-summary} summarizes win/tie/loss counts, gap statistics, and d-stationarity failures.


\paragraph{Trimmed lasso.}
In trimmed lasso, the difference between criticality and d-stationarity is evident in our diagnostics. In the focused setting
($m=50$, $n=100$, $k=5$, $\lambda=1$), plain DCA fails the d-stationarity
diagnostic on $99/100$ instances versus $4/100$ with augmentation, with objective outcomes $99/1/0$ wins/ties/losses under $\Delta$.
Directions are drawn from the gauss-axis mixture specified in Appendix~\ref{sec:appendix-experiments-v2}, which also records matched sphere runs under the same budgets. Uniform sphere sampling accepts virtually no exploration steps and behaves similarly to DCA.
The high base failure rate alongside the strong augmentation wins is consistent with a standard DC picture \citep{deOliveira2020}: where the nonsmooth term has nonsingleton subgradient, criticality ($\partial f(x)\cap\partial g(x)\neq\emptyset$) can hold without d-stationarity ($\partial g(x)\subseteq\partial f(x)$).

\paragraph{Least trimmed squares.}
Focused LTS ($m=100$, $n=50$, $q=10$): zero d-stationarity failures for base and augmented runs, $9/91/0$ wins/ties/losses, and a large win-conditional median $\Delta$ ($+52.2$). Augmentation is mostly inert but never hurts and occasionally finds a substantially better basin.

\paragraph{ReLU regression.}
Two focused ReLU settings ($200$ comparisons) with a prox-linear inner oracle and an exact DC subdifferential-inclusion d-stationarity check: zero d-stationarity failures, $44/156/0$ wins/ties/losses, mean $\Delta{=}+0.96$, McNemar $p{=}5.7\cdot 10^{-14}$, showing the same augmentation carries beyond DCA.

\begin{table}[t]
\centering
\caption{Headline outcomes across the three focused benchmarks.
\emph{w/t/l} denotes win/tie/loss under
$\Delta = h_{\mathrm{base}} - \mathrm{median}_s\, h_{\mathrm{aug},s}$.
\emph{\# non-d-stat (base / aug)} counts runs that fail the benchmark
d-stationarity diagnostic. ReLU aggregates two settings ($200$ paired
comparisons total).}
\label{tab:experiments-story-summary}
\small
\setlength{\tabcolsep}{4pt}
\begin{tabular}{l c c c c c c}
\toprule
Experiment & $n_{\mathrm{inst}}$ & w/t/l & mean $\Delta$ & med.\ win $\Delta$ & McNemar $p$ & \#~non-d-stat (base / aug) \\
\midrule
Trimmed lasso, $\lambda{=}1$        & 100 & 99/1/0    & $+0.29$ & $+0.27$ & $1.6 \cdot 10^{-30}$ & $99 / 4$ \\
LTS, $\sigma_{\mathrm{clean}}{=}4$  & 100 & 9/91/0    & $+3.96$ & $+52.2$ & $2.0 \cdot 10^{-3}$  & $0 / 0$  \\
ReLU regression                     & 200 & 44/156/0  & $+0.96$ & $+1.26$ & $5.7 \cdot 10^{-14}$ & $0 / 0$  \\
\bottomrule
\end{tabular}
\end{table}


\section{Conclusions and Limitations}\label{sec:limitations}\label{sec:conclusion}
\paragraph{Conclusions.} We showed that one additional random feasible-direction step per iteration suffices to enforce almost-sure convergence of accumulation points to d-stationarity, independently of the underlying oracle. 
The key mechanism is a structural separation: the oracle determines convergence rates and practical performance, while exploration enforces asymptotic d-stationarity at negligible cost. 
Under the conditions of Theorems~\ref{thm:rate-transfer} and~\ref{thm:certificate}, rate guarantees of the base method are preserved, with the exploration step entering only through the inequality $h(x^{k+1}) \leq h(z^{k+1})$, while d-stationarity follows from the exploration step alone. 
Rate transfer covers DCA and prox-linear methods through Theorem~\ref{thm:rate-transfer}, and gradient sampling through Theorem~\ref{thm:certificate}. 
These results suggest that lightweight directional exploration can serve as a modular tool for strengthening optimality guarantees. 
In particular, integrating structured feasible-direction mechanisms with model-based updates is a promising direction for advancing nonsmooth nonconvex optimization.

\paragraph{Limitations.} The d-stationarity guarantee is asymptotic. Theorem~\ref{thm:random} ensures almost-sure d-stationarity of accumulation points, but provides no finite-iteration bound for approaching or certifying it. The rates in Section~\ref{sec:oracles} concern the stationarity measure supplied by the base oracle, not d-stationarity; thus finite-time proximity to a d-stationary point is not guaranteed.

Finite-time performance depends on the sampling distribution $\mathcal{D}_v$. Although the theory requires only full support on $S^{n-1} \cap L$, this may be ineffective in practice. In trimmed lasso, the uniform sphere sampler accepts virtually no exploration steps, whereas an axis-biased Gaussian mixture yields consistent gains by exploiting $\ell_1$-type sparsity. 
When such structure is unavailable and the base oracle is treated as a black box, designing an effective sampler is nontrivial.

\begin{ack}
No external funding was received for this work. The authors declare no competing financial interests.
\end{ack}

\bibliographystyle{plainnat}
\bibliography{references}

\newpage
\appendix
\section{Proofs for Sections~\ref{sec:intro} and~\ref{sec:convergence}}
\label{sec:appendix Proofs of}

\begin{proof}[Proof of Lemma~\ref{ass:A3}]
By hypothesis there is $R > 0$ such that $h$ is $L$-Lipschitz on the open ball $B(x,R) := \{y \in \R^n : \norm{y - x} < R\}$. Fix $d_1, d_2 \in \R^n$ and $0 < t < R/(1 + \max\{\norm{d_1}, \norm{d_2}\})$. Then $x + td_1, x + td_2 \in B(x,R)$, and therefore
\begin{equation*}
\left| \frac{h(x+td_1) - h(x)}{t} - \frac{h(x+td_2) - h(x)}{t} \right|
= \frac{|h(x+td_1) - h(x+td_2)|}{t}
\leq L\norm{d_1 - d_2}.
\end{equation*}
By Assumption~A2, since $x \in C$, the two difference quotients converge to $h'(x;\,d_1)$ and $h'(x;\,d_2)$ as $t \downarrow 0$, and the bound passes to the limit. Continuity of $d \mapsto h'(x;\,d)$ follows.
\end{proof}

Figure~\ref{fig:rep-mechanism} illustrates part~(ii) of the proof of Theorem~\ref{thm:random}.

\begin{figure}[!ht]
\centering
\includegraphics[width=\linewidth]{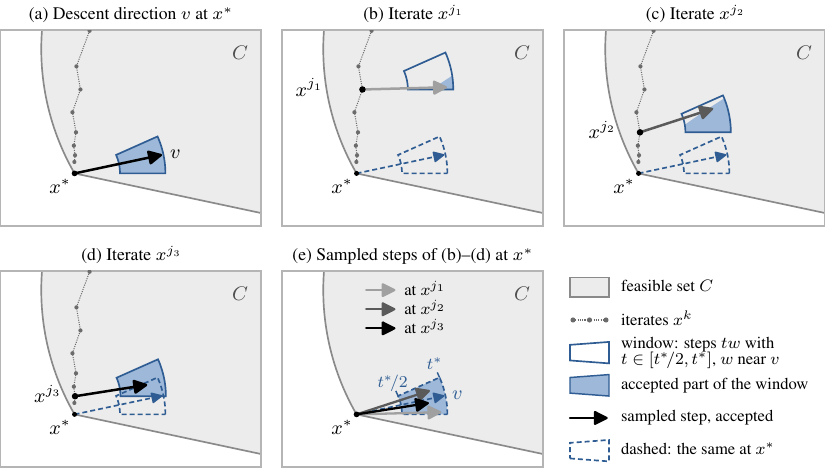}
\caption{The contradiction argument in part~(ii) of the proof of Theorem~\ref{thm:random}, in a two-dimensional example with $L = \R^2$ and a nonsmooth $h$. (a) Suppose, for contradiction, that an accumulation point $x^*$ of $\{x^k\}$ admits a feasible descent direction. As in the proof, by Lemma~\ref{lem:continuity} and continuity of $h'(x^*;\,\cdot)$, $v$ can be taken among the directions $v_i$ there, which lie in the relative interior of $D_C(x^*)$, with $h'(x^*;\,v) < 0$; the arrow shows $t^* v$. The window at a point $x$ consists of the points $x + tw$ with $t \in [t^*/2, t^*]$ and $w$ a unit direction sufficiently close to $v$. Its shaded part consists of the points $y \in C$ that pass the acceptance test $h(y) + \frac{\gamma}{2}\norm{y - x}^2 < h(x)$ of Algorithm~\ref{alg:ERP} at $x$. At $x^*$, the whole window passes. (b)--(d) Iterates $x^{j_1}, x^{j_2}, x^{j_3}$ along the subsequence of Lemma~\ref{lem:BC-directions}, with the sampled step $\hat t^j v^j_{\mathrm{ex}}$ at each (arrow). The dashed outline is the window at $x^*$. As $x^j \to x^*$, the window at $x^j$ approaches the window at $x^*$, and for $x^j$ sufficiently close to $x^*$, the whole window at $x^j$ lies in $C$ (Lemma~\ref{lem:step-bound}) and passes the acceptance test (the estimates in the proof of Theorem~\ref{thm:random}). (e) The three sampled steps, redrawn from $x^*$ at the same scale, inside the window at $x^*$ (dashed). Along the subsequence, $v^j_{\mathrm{ex}} \to v$ and $\hat t^j \in [t^*/2, t^*]$, so for all large $j$ the sampled step falls inside the window. Once the whole window at $x^j$ passes the test, each such step is accepted, with $t^j_{\mathrm{ex}} = \hat t^j \ge t^*/2$. This happens for infinitely many $j$, so the accepted step lengths do not tend to zero: $\sum_k (t^k_{\mathrm{ex}})^2 = \infty$, which contradicts part~(i). Hence $x^*$ admits no feasible descent direction.}
\label{fig:rep-mechanism}
\end{figure}

\begin{proof}[Proof of Lemma~\ref{lem:BC-directions}]
We assume throughout that $\dim L \ge 1$; otherwise $C$ is a singleton and the result is trivial.

Recall the $\sigma$-algebra $\mathcal{F}_{k-1}$ of Remark~\ref{rem:sampling}: $x^k$ is $\mathcal{F}_{k-1}$-measurable, while $v^k_{ex}$ and $\hat{t}^k$ are independent of $\mathcal{F}_{k-1}$ and of each other.

Fix a point $y \in C$, a radius $\delta > 0$, a spherical cap
$V := B[w, \varepsilon'] \cap S^{n-1} \cap L,$
for some $w \in S^{n-1} \cap L$ and $\varepsilon' > 0$, and an interval $[a', b'] \subseteq (0, r]$ with $a' < b'$. 
Set
\begin{equation*}
p_v := \mathbb{P}(v^k_{ex} \in V) > 0,\qquad
p_t := \mathbb{P}(\hat t^k \in [a', b']) > 0,
\end{equation*}
where positivity follows from the full-support assumptions.

Define
\[
A_k := \{v^k_{ex} \in V\} \cap \{\hat{t}^k \in [a', b']\} \cap \{\norm{x^k - y} < \delta\}.
\]
Independence yields $\mathbb{P}(A_k \mid \mathcal{F}_{k-1})
= \mathbf{1}_{\{\norm{x^k - y} < \delta\}} \cdot p_v \cdot p_t.$

By the conditional Borel--Cantelli--L\'evy lemma \citep[Ch.~VII, \S5, Cor.~2]{Shiryaev2019},
\begin{equation*}
  \Bigl\{\sum_{k} \mathbb{P}(A_k \mid \mathcal{F}_{k-1}) = \infty\Bigr\}
  =
  \{A_k\ \text{occurs infinitely often}\}
  \quad\text{a.s.}
\end{equation*}
In particular, whenever $x^k$ visits $B(y,\delta)$ infinitely often, $A_k$ occurs infinitely often almost surely.

Let $\{y_n\}$ be a countable dense subset of $C$, let $\mathcal{Q}=\{q_i\}_{i\in\N}$, and let $(\delta_m, \varepsilon'_l, a'_s, b'_s)$ range over positive rationals with $a'_s < b'_s \leq r$.
For each tuple $u=(n,i,m,l,s)$, define the event
\[
E_u := \Bigl\{\|x^k-y_n\|<\delta_m\ \text{i.o.}\ \Longrightarrow\
\bigl(v^k_{ex}\in B[q_i,\varepsilon'_l]\cap S^{n-1}\cap L,\ \hat t^k\in[a'_s,b'_s]\bigr)\ \text{i.o.}\Bigr\}.
\]
By the conditional Borel--Cantelli--L\'evy lemma, $\mathbb{P}(E_u)=1$ for every $u$: on paths where $\|x^k-y_n\|<\delta_m$ occurs infinitely often, the enriched event occurs infinitely often almost surely; otherwise the implication is vacuous.
Set $\Omega_{BC}:=\bigcap_{u} E_u.$
Since the index set $\{u\}$ is countable and each $E_u$ has probability $1$, we have $\mathbb{P}(\Omega_{BC})=1$. 
On $\Omega_{BC}$, the following grid property holds simultaneously:

For every tuple $(y_n, \delta_m, q_i, \varepsilon'_l, a'_s, b'_s)$, if $\norm{x^k - y_n} < \delta_m$ for infinitely many $k$, then there are infinitely many such $k$ with additionally
\[
v^k_{ex} \in B[q_i, \varepsilon'_l] \cap S^{n-1} \cap L,
\qquad
\hat{t}^k \in [a'_s, b'_s].
\]

Fix $\omega \in \Omega_{BC}$. Let $x^*$ be any accumulation point of $\{x^k(\omega)\}$ and fix $v \in \mathcal{Q}$. For this fixed $\omega$, write $x^k:=x^k(\omega)$ and similarly for sampled variables. Fix arbitrary $\eta \in (0,\rho)$. By density, choose rationals $y_n, \delta_m, \varepsilon'_l, a'_s, b'_s$ such that
$
\|y_n - x^*\| < \eta/4,
\delta_m \in (\eta/2, 3\eta/4),
\varepsilon'_l < \eta/4,
$
and
$
a < a'_s < b'_s < b.
$
Choose $i$ such that $q_i=v$.
Since $x^*$ is an accumulation point, $\|x^k - x^*\| < \eta/4$ for infinitely many $k$. For such indices,
\[
\|x^k - y_n\|
\le \|x^k - x^*\| + \|x^* - y_n\|
< \eta/2 < \delta_m.
\]
Hence, by the grid property, there are infinitely many indices $k$ such that: $\|x^k - y_n\| < \delta_m, v^k_{ex} \in B[q_i, \varepsilon'_l],$ for $\hat{t}^k \in [a'_s, b'_s].$
For these indices,
\[
\|x^k - x^*\|
\le \|x^k - y_n\| + \|y_n - x^*\|
< \delta_m + \eta/4
< \eta,
\]
and $\|v^k_{ex} - v\|
< \varepsilon'_l
< \eta < \rho,$
while $\hat{t}^k \in [a'_s, b'_s] \subseteq (a,b)$.

Let $\eta_q \downarrow 0$ with $\eta_1 < \rho$, and for each $q$ let $K_q$ be the infinite set of indices satisfying the above properties with $\eta=\eta_q$. Choose inductively $j_q \in K_q$ with $j_q > j_{q-1}$. Then, since $\|x^{j_q} - x^*\| < \eta_q$ and $\eta_q \to 0$, we have $x^{j_q} \to x^*$. Similarly, $\|v^{j_q}_{ex} - v\| < \eta_q$ implies $v^{j_q}_{ex} \to v$, while $\|v^{j_q}_{ex} - v\| < \rho$ and $\hat{t}^{j_q} \in [a,b]$ for all $q$. This completes the proof.
\end{proof}

\begin{proof}[Proof of Lemma \ref{lem:step-bound}]
Let $k := \dim(L)$. Since $v \in \ri(D_C(x^*))$, there exists $r_0 > 0$ such that
\[
B(v, r_0) \cap L \subseteq \ri(D_C(x^*)).
\]
Pick $k+1$ directions $d_0, \ldots, d_k \in B(v, r_0) \cap L$ whose convex hull is a $k$-dimensional simplex in $L$ containing $v$ in its relative interior. Then
\[
\conv\{d_0, \ldots, d_k\} \subseteq \ri(D_C(x^*)).
\]
Since each $d_j \in D_C(x^*)$, choose $\varepsilon_j > 0$ with $x^* + \varepsilon_j d_j \in C$, and set $\bar{\varepsilon} := \min_j \varepsilon_j > 0$. By convexity of $C$, $p_j := x^* + \bar{\varepsilon}\, d_j \in C$ for each $j$, and for any convex combination $d = \sum_j \lambda_j d_j$,
\[
  x^* + \bar{\varepsilon}\, d = \sum_{j} \lambda_j\, p_j \in C.
\]
Since $v$ lies in the relative interior of the simplex, there exists $\rho_0 > 0$ with
\[
B(v, \rho_0) \cap L \subseteq \conv\{d_0, \ldots, d_k\},
\]
and hence $x^* + \bar{\varepsilon}\, w \in C$ for every $w \in B(v, \rho_0) \cap L$. The map $w \mapsto x^* + \bar{\varepsilon}\, w$ sends $B(v,\rho_0)\cap L$ onto
\[
B(x^* + \bar{\varepsilon}\, v,\; \bar{\varepsilon}\rho_0)\cap \aff(C),
\]
and therefore
\[
  B\!\left(x^* + \bar{\varepsilon}\, v,\; \bar{\varepsilon}\rho_0\right) \cap \aff(C) \subseteq C.
\]

Set $\rho := \rho_0 / 2$ and $\delta := \bar{\varepsilon}\rho_0 / 2$. Then (i) holds because
\[
B(v,\, \rho) \cap L \subseteq B(v,\, \rho_0) \cap L \subseteq \conv\{d_0, \ldots, d_k\} \subseteq \ri(D_C(x^*)).
\]
For (ii), fix $x \in C$ with $\norm{x - x^*} < \delta$ and $w \in B(v,\, \rho) \cap L$. Then
\[
  \norm{(x + \bar{\varepsilon}\, w) - (x^* + \bar{\varepsilon}\, v)} \leq \norm{x - x^*} + \bar{\varepsilon}\norm{w - v} < \delta + \bar{\varepsilon}\rho = \bar{\varepsilon}\rho_0,
\]
and since $x \in C \subseteq \aff(C)$ and $w \in L$, we have $x + \bar{\varepsilon}\, w \in \aff(C)$; hence
\[
x + \bar{\varepsilon}\, w \in B(x^* + \bar{\varepsilon}\, v,\, \bar{\varepsilon}\rho_0) \cap \aff(C) \subseteq C.
\]
Convexity then gives
\[
  x + t w = (1 - t/\bar{\varepsilon})\, x + (t/\bar{\varepsilon})\,(x + \bar{\varepsilon}\, w) \in C \qquad \text{for every } t \in [0,\, \bar{\varepsilon}],
\]
so $\tau(x;\, w) \geq \bar{\varepsilon}$.
\end{proof}

\medskip

\begin{proof}[Proof of Lemma \ref{lem:continuity}]
Let $K := D_C(x^*)$, a convex cone. 
By Lemma~\ref{ass:A3}, the map $d \mapsto h'(x^*;\, d)$ is continuous.
We next show positive homogeneity. For any $\alpha > 0$,
\[
  h'(x^*;\, \alpha d)
  = \lim_{t\downarrow 0}\frac{h(x^* + t\alpha d)-h(x^*)}{t}
  = \alpha \lim_{t\downarrow 0}\frac{h(x^* + t d)-h(x^*)}{t}
  = \alpha h'(x^*;\, d).
\]

For the extension statement, recall that for any convex set $K$,
\[
\cl(\ri(K)) = \cl(K),
\]
where the relative interior is taken with respect to $\aff(K)=L$. Let $d \in \cl(K)$. Then there exists a sequence $v_m \in \ri(K)$ with $v_m \to d$. By continuity,
\[
  h'(x^*;\, d) = \lim_{m\to\infty} h'(x^*;\, v_m) \geq 0,
\]
since $h'(x^*;\, v_m) \geq 0$ for all $m$.
\end{proof}

\section{Proofs for Section~\ref{sec:dc}: DCA}
\label{sec:appendix-dca}

\begin{lemma}[Descent at the oracle output]\label{lem:dca-descent}
For each $k$, $h(x^k) - h(z^{k+1}) \geq \frac{\mu_1 + \mu_2}{2}\norm{x^k - z^{k+1}}^2$.
\end{lemma}

\begin{proof}
By $\mu_1$-strong convexity of $f_1$ at the subgradient $s^{k+1} \in \partial f_1(z^{k+1})$:
\[
  f_1(x^k) \geq f_1(z^{k+1}) + \inner{s^{k+1}}{x^k - z^{k+1}} + \tfrac{\mu_1}{2}\norm{x^k - z^{k+1}}^2.
\]
First-order optimality of the DCA subproblem at $z^{k+1}$ gives $s^{k+1} \in \partial f_1(z^{k+1})$, $n^{k+1} \in N_C(z^{k+1})$, and
\begin{equation}\label{eq:OC}
  s^{k+1} + n^{k+1} = g^k_2.
\end{equation}
Substituting $s^{k+1} = g^k_2 - n^{k+1}$ from~\eqref{eq:OC} and using $\inner{n^{k+1}}{x^k - z^{k+1}} \leq 0$ (since $n^{k+1} \in N_C(z^{k+1})$ and $x^k \in C$):
\begin{equation}\label{eq:f1-simplified}
  f_1(x^k) \geq f_1(z^{k+1}) + \inner{g^k_2}{x^k - z^{k+1}} + \tfrac{\mu_1}{2}\norm{x^k - z^{k+1}}^2.
\end{equation}
By $\mu_2$-strong convexity of $f_2$ with subgradient $g^k_2 \in \partial f_2(x^k)$:
\begin{equation}\label{eq:f2-bound}
  f_2(x^k) \leq f_2(z^{k+1}) + \inner{g^k_2}{x^k - z^{k+1}} - \tfrac{\mu_2}{2}\norm{x^k - z^{k+1}}^2.
\end{equation}
Subtracting~\eqref{eq:f2-bound} from~\eqref{eq:f1-simplified}, the $\inner{g^k_2}{\cdot}$ terms cancel:
$h(x^k) - h(z^{k+1}) \geq \frac{\mu_1 + \mu_2}{2}\norm{x^k - z^{k+1}}^2$.
\end{proof}

\begin{proof}[Proof of Theorem~\ref{thm:dca-rate}]
By greedy selection, $h(x^{k+1}) \leq h(z^{k+1})$. Combining with Lemma~\ref{lem:dca-descent}:
\[
  h(x^k) - h(x^{k+1}) \geq \frac{\mu_1 + \mu_2}{2}\norm{x^k - z^{k+1}}^2.
\]
Summing over $k = 0, \ldots, N-1$ and telescoping:
\[
  \frac{\mu_1 + \mu_2}{2}\sum_{k=0}^{N-1}\norm{x^k - z^{k+1}}^2 \leq h(x^0) - h(x^N) \leq h(x^0) - h^*.
\]
Taking the minimum over $k$ (at most the average) and the square root completes the proof.
\end{proof}

\begin{remark}[Oracle residual $\norm{x^k - z^{k+1}}$]\label{rem:pl-oracle-residual}\label{rem:dca-optimality-measure}
For the DCA oracle, $\norm{x^k - z^{k+1}} = 0$ if and only if $x^k$ is critical (a fixed point of the oracle subproblem). Thus $\norm{x^k - z^{k+1}}$ is a quantitative residual for distance to criticality, and Theorem~\ref{thm:dca-rate} gives the preserved ergodic $O(1/\sqrt{N})$ rate in exactly this residual.

This rate statement concerns a finite-iteration optimality measure, whereas the asymptotic guarantee of the augmented method (Theorem~\ref{thm:random}) is almost-sure d-stationarity of accumulation points. There is no conflict: the base algorithm converges to criticality, while augmentation upgrades the limit notion to d-stationarity, which is strictly stronger in general (Remark~\ref{rem:stationarity-comparison}).
\end{remark}
\section{Proofs for Section~\ref{sec:rate-transfer}: rate transfer}\label{sec:appendix-rate-transfer}

\begin{proof}[Proof of Proposition~\ref{prop:budget}]
By \eqref{eq:greedy-descent}, $h(x^k) - h(z^{k+1}) \leq h(x^k) - h(x^{k+1})$ for every $k$. Summing over $k = 0, \dots, K-1$ telescopes the right-hand side to $h(x^0) - h(x^K) \leq \Delta_0$, since $x^K \in C$ and $h \geq h^*$ on $C$. The second bound holds because a minimum is at most an average. For the third, the summands are nonnegative along the run, so $\eta \cdot \#\{k<K : h(x^k) - h(z^{k+1}) \geq \eta\} \leq \sum_{k=0}^{K-1}\big(h(x^k) - h(z^{k+1})\big) \leq \Delta_0$.
\end{proof}

\begin{proof}[Proof of Theorem~\ref{thm:rate-transfer}]
Write $\delta_k := h(x^k) - h(z^{k+1})$, which is nonnegative by \eqref{eq:trajectory}. Let $k^* \in \arg\min_{k<K} \delta_k$. By Proposition~\ref{prop:budget}, $\delta_{k^*} \leq \Delta_0/K$, and by \eqref{eq:trajectory} and monotonicity of $\varphi$,
\begin{equation*}
  \min_{k<K} \sigma(x^k, z^{k+1}) \leq \sigma(x^{k^*}, z^{k^*+1}) \leq \varphi(\delta_{k^*}) \leq \varphi(\Delta_0/K).
\end{equation*}
For the second claim, let $k < K$ satisfy $\sigma(x^k, z^{k+1}) > \varphi(\eta)$. If $\delta_k \leq \eta$, then \eqref{eq:trajectory} and monotonicity of $\varphi$ give $\sigma(x^k, z^{k+1}) \leq \varphi(\delta_k) \leq \varphi(\eta)$, a contradiction. Hence $\delta_k > \eta$, and Proposition~\ref{prop:budget} bounds the number of such $k$ by $\Delta_0/\eta$.
\end{proof}

\begin{proof}[Proof of Theorem~\ref{thm:certificate}]
Since (P3) also holds with any larger $q < 1/e$, we may assume $q > 0$. Fix $c > eq/(1-eq)$ depending only on $q$ (for instance $c := 2eq/(1-eq)$, for which $\rho = (1+eq)/2$ below), and set
\begin{equation*}
  \rho := e(1+1/c)q, \qquad T := \lceil \Delta_0/\bar\eta \rceil, \qquad m := \max\big\{\lceil cT \rceil, \lceil \log(1/\zeta)/\log(1/\rho) \rceil\big\}, \qquad N := T + m .
\end{equation*}
We show that, with probability at least $1-\zeta$, some call $\mathcal{O}_C(x^k)$ with $k < N$ returns $\mathrm{CERT}$. Since $0 < eq < 1$, the condition $c > eq/(1-eq)$ is equivalent to $e(1+1/c)q < 1$, so $\rho \in (0,1)$, and $m \geq 1$ because $\zeta < 1$. By (P2), every $k < N$ with $\ell^k = \mathrm{PROG}$ satisfies $h(x^k) - h(z^{k+1}) \geq \bar\eta$. By (O3) and the third bound of Proposition~\ref{prop:budget} with $\eta = \bar\eta$, there are at most $\Delta_0/\bar\eta \leq T$ such $k$, on every realization. Let $E$ be the event that no call $\mathcal{O}_C(x^k)$ with $k < N$ returns $\mathrm{CERT}$. On $E$, each of these $N$ calls returns $\mathrm{PROG}$ or $\mathrm{FAIL}$, so at least $N - T = m$ of them return $\mathrm{FAIL}$. Hence $E \subseteq \bigcup_S E_S$, where $S$ ranges over the subsets of $\{0, \dots, N-1\}$ with $m$ elements and $E_S := \{\ell^k = \mathrm{FAIL} \text{ for every } k \in S\}$.

Fix $S = \{k_1 < \dots < k_m\}$. The event $\{\ell^{k_i} = \mathrm{FAIL},\ i < m\}$ belongs to $\mathcal{G}_{k_m - 1}$, since every such $k_i$ satisfies $k_i \leq k_m - 1$. By the tower property and (P3),
\begin{equation*}
  \mathbb{P}(E_S) = \mathbb{E}\Big[\mathbf{1}_{\{\ell^{k_i} = \mathrm{FAIL},\ i < m\}}\, \mathbb{P}\big(\ell^{k_m} = \mathrm{FAIL} \mid \mathcal{G}_{k_m - 1}\big)\Big] \leq q\, \mathbb{P}\big(\ell^{k_i} = \mathrm{FAIL},\ i < m\big),
\end{equation*}
and iterating down through $k_{m-1}, \dots, k_1$ gives $\mathbb{P}(E_S) \leq q^m$. A union bound over the $\binom{N}{m}$ sets $S$ and the bound $\binom{N}{m} \leq (eN/m)^m$ give $\mathbb{P}(E) \leq (eNq/m)^m$. Since $m \geq cT$, we have $N = T + m \leq (1 + 1/c)m$, and therefore
\begin{equation*}
  \mathbb{P}(E) \leq \big(e(1+1/c)q\big)^m = \rho^m \leq \zeta,
\end{equation*}
the last inequality because $m \geq \log(1/\zeta)/\log(1/\rho)$. Finally, $\lceil a \rceil \leq a + 1$ and $\max\{a, b\} \leq a + b$ for $a, b \geq 0$ give
\begin{equation*}
  N \leq (1+c)T + \frac{\log(1/\zeta)}{\log(1/\rho)} + 2 \leq (1+c)\frac{\Delta_0}{\bar\eta} + \frac{\log(1/\zeta)}{\log(1/\rho)} + c + 3 \leq \kappa_q\Big(1 + \frac{\Delta_0}{\bar\eta} + \log(1/\zeta)\Big),
\end{equation*}
with $\kappa_q := \max\{c + 3,\ 1/\log(1/\rho)\}$, which depends only on $q$.
\end{proof}

\subsection{Scope of rate transfer}\label{subsec:rate-transfer-scope}

Theorems~\ref{thm:rate-transfer} and~\ref{thm:certificate} give sufficient conditions for rate transfer. Both use the exploration step only through \eqref{eq:greedy-descent}, which bounds $h(x^{k+1})$ by $h(z^{k+1})$ but does not give $x^{k+1} = z^{k+1}$. The rate analysis of a base method therefore yields \eqref{eq:trajectory} along a run of Algorithm~\ref{alg:random} only if its descent inequality holds at every point passed to the oracle, including points produced by exploration, and the same holds for (P3). In the cases below, Theorems~\ref{thm:rate-transfer} and~\ref{thm:certificate} do not establish that the rate of the base method is preserved. We do not claim that it is lost.

\paragraph{Oracles that carry state.}
Condition (O1) fails when the output of the oracle at $x$ depends on state carried across calls. Accelerated methods are the main example: FISTA \citep{Beck2009b} takes a proximal gradient step from a point extrapolated from its two previous iterates, and its $O(1/k^2)$ analysis uses that each iterate is this step. Bundle methods build their model from cutting planes collected at earlier iterates, and quasi-Newton methods update a curvature matrix from differences of successive iterates and gradients. For such an oracle, $\mathcal{O}_C(x)$ is not determined by $x$, so $\Lambda$ is not a function on $C$ and \eqref{eq:calibration} is not well posed. Theorem~\ref{thm:rate-transfer} still holds along any run that satisfies \eqref{eq:trajectory}. The analysis of the base method, however, proves its descent inequality along its own trajectory and with its own state, and does not in general supply \eqref{eq:trajectory} at an iterate produced by exploration.

With a descent inequality proved at every point of $C$, condition (O1) yields \eqref{eq:trajectory} at such iterates. It is not necessary for \eqref{eq:trajectory}. A prox-linear oracle fails (O1) when its inner solver is warm-started at the previous output and the point it returns depends on that output. If every returned point still satisfies \eqref{eq:pl-model-decrease}, then for $\rho > L_c L_F$ Lemma~\ref{lem:pl-descent} gives $h(x^k) - h(z^{k+1}) \geq \alpha \norm{x^k - z^{k+1}}^2$ with $\alpha = \tfrac{1}{2}(\rho - L_c L_F) > 0$. That decrease is the quadratic half of the hypothesis of Theorem~\ref{thm:pl-rate}, and it is \eqref{eq:trajectory} for $\sigma(x, z) = \norm{x - z}$ and $\varphi(d) = \sqrt{d/\alpha}$, along every run.

For Theorem~\ref{thm:certificate}, (O1) is what turns a failure bound at every fixed point of $C$ into (P3). The single-pass INGD oracle of Section~\ref{subsec:gs-oracle} restarts its inner loop at each call, and Lemma~\ref{lem:ingd-fail} bounds its failure probability at every fixed anchor. Lemma~\ref{lem:ingd-fail} does not cover a variant that carries the vector $m_j$ across calls.

Theorem~\ref{thm:random} imposes no such condition. It allows randomized and adaptive oracles, provided each exploration sample $(v^k_{ex}, \hat{t}^k)$ is independent of the $\sigma$-algebra $\mathcal{F}_{k-1}$, as assumed in Remark~\ref{rem:sampling}.

Monotone accelerated methods are a boundary case. MFISTA \citep{Beck2009a} takes as its next iterate the better, by function value, of the proximal gradient step and the previous iterate, and it retains the $O(1/k^2)$ rate of FISTA for convex problems. Its output depends on the extrapolation point carried from earlier iterates, and so it fails (O1). Whether its analysis admits the exploration point as a further candidate is open.

\paragraph{Iterates feasible only in the limit.}
Condition (O2) fails for methods whose iterates are feasible only in the limit, such as penalty, augmented Lagrangian and ADMM-type methods. Algorithm~\ref{alg:random} requires $z^{k+1} \in C$, so such a method lies outside the scheme, and Theorem~\ref{thm:random}, which also assumes feasible oracle outputs, does not apply to it. Proposition~\ref{prop:budget} uses feasibility as well, through $h(x^K) \geq h^*$. Moreover, the analysis of such a method controls a merit function, such as a penalty or augmented Lagrangian function, rather than $h$, so (O3) need not hold.    Whether an exploration step can be combined with such methods, for example through their merit function, is open.

\paragraph{Measures not calibrated to the decrement.}
Condition \eqref{eq:calibration} constrains the pair $(\mathcal{O}_C, \sigma)$. It is required at every $x \in C$ and, for a randomized oracle, for every realization of $\xi$. The residual $\norm{x - z}$ of Sections~\ref{sec:dc} and~\ref{subsec:pl-oracle} admits a calibration because Lemmas~\ref{lem:dca-descent} and~\ref{lem:pl-descent} bound $\Lambda$ below by a multiple of its square. A measure such as the distance to a solution set can admit no calibration even when $\mathcal{O}_C$ satisfies (O1)--(O3), as the following example shows.

\begin{remark}\label{rem:scope-example}
Let $n = 1$, $C = \R$ and $h(x) := \min\{|x|, 1\}$. Then $h$ is $1$-Lipschitz, directionally differentiable and bounded below, so Assumptions~A1 and~A2 hold, and $\arg\min h = \{0\}$. Let $\mathcal{O}_C(x) := 0$ if $|x| \leq 1$, and $\mathcal{O}_C(x) := x$ if $|x| > 1$, where $h$ is constant. Then $\mathcal{O}_C$ is a map on $C$, so (O1) holds, and (O2) holds since $C = \R$. Condition (O3) holds because $h(\mathcal{O}_C(x)) \leq h(x)$ for every $x$: if $|x| \leq 1$, then $\mathcal{O}_C(x) = 0$ and $h(0) = 0 \leq h(x)$, and if $|x| > 1$, then $\mathcal{O}_C(x) = x$ and $h(\mathcal{O}_C(x)) = h(x)$. Let $\sigma(x, z) := \mathrm{dist}(x, \arg\min h) = |x|$. For every $x$ with $|x| > 1$, $\Lambda(x) = 0$ and $\sigma(x, \mathcal{O}_C(x)) = |x|$. Hence \eqref{eq:calibration} would require $|x| \leq \varphi(0)$ for every such $x$, which no $\varphi \colon \R_{\geq 0} \to \R_{\geq 0}$ satisfies. Thus \eqref{eq:calibration} is not implied by (O1)--(O3). The base method with this oracle has no rate in $\sigma$ either when $|x^0| > 1$, so the example shows that \eqref{eq:calibration} restricts the pair $(\mathcal{O}_C, \sigma)$, not that augmentation loses a rate.
\end{remark}

\paragraph{Nonmonotone methods.}
Condition (O3) fails for base methods whose iterations may increase $h$, such as nonmonotone line-search methods. The non-ascent part of \eqref{eq:trajectory} is used by the third bound of Proposition~\ref{prop:budget} and by Theorem~\ref{thm:rate-transfer}, whose $\varphi$ is defined on $\R_{\geq 0}$, and Theorem~\ref{thm:certificate} uses (O3) through the same bound. Greedy selection does not restore non-ascent of the oracle output: it gives $h(x^{k+1}) \leq h(y^{k+1}) \leq h(x^k)$, not $h(z^{k+1}) \leq h(x^k)$.


\section{Proofs for Section~\ref{subsec:pl-oracle}: prox-linear-type oracles}
\label{sec:appendix-pl}

This appendix supports Section~\ref{subsec:pl-oracle}.
Throughout, $z^{k+1}$ denotes the candidate produced by a prox-linear-type oracle at the current iterate $x^k$ of Algorithm~\ref{alg:random} (unconstrained $C=\R^n$ in~\eqref{eq:prox-linear}).
As in Section~\ref{subsec:pl-oracle}, $c \colon \R^m \to \R$ is globally Lipschitz, with constant $L_c$, and $F \colon \R^n \to \R^m$ is continuously differentiable, with $\nabla F$ Lipschitz with constant $L_F$ in the operator norm.
We write $\ell^k(y) := F(x^k) + \nabla F(x^k)(y-x^k)$ and $m^k(y) := c(\ell^k(y)) + \tfrac{\rho}{2}\norm{y-x^k}^2$.
The descent calculation is classical for Lipschitz composite models \citep{Drusvyatskiy2019}.

\subsection{Lemma~\ref{lem:pl-descent} and Theorem~\ref{thm:pl-rate}}\label{subsec:app-pl-descent-and-rate}

\begin{lemma}[Quadratic descent from model decrease]\label{lem:pl-descent}
Under the Lipschitz assumptions on $c$ and $F$ and with $\rho > L_c L_F$, if~\eqref{eq:pl-model-decrease} holds at iteration $k$, then
\begin{equation}\label{eq:pl-qdescent-state}
  h(x^k) - h(z^{k+1}) \;\geq\; \tfrac{1}{2}(\rho - L_c L_F)\,\norm{z^{k+1} - x^k}^2.
\end{equation}
\end{lemma}

\begin{proof}
From $m^k(z^{k+1}) \leq m^k(x^k) = h(x^k)$,
\begin{equation}\label{eq:app-pl-step1}
  h(x^k) - c(\ell^k(z^{k+1})) \;\geq\; \tfrac{\rho}{2}\norm{z^{k+1}-x^k}^2.
\end{equation}
By the Lipschitz Jacobian estimate for $F$ \citep{Drusvyatskiy2019},
$\norm{F(z^{k+1}) - \ell^k(z^{k+1})} \leq \tfrac{L_F}{2}\norm{z^{k+1}-x^k}^2$, and Lipschitzness of $c$ gives
$c(\ell^k(z^{k+1})) \geq c(F(z^{k+1})) - L_c \norm{F(z^{k+1})-\ell^k(z^{k+1})}
\geq h(z^{k+1}) - \tfrac{L_c L_F}{2}\norm{z^{k+1}-x^k}^2$.
Substituting into~\eqref{eq:app-pl-step1} yields~\eqref{eq:pl-qdescent-state}.
\end{proof}

\begin{proof}[Proof of Theorem~\ref{thm:pl-rate}]
By hypothesis, $h(z^{k+1}) \leq h(x^k)$ and $\norm{x^k - z^{k+1}} \leq \sqrt{(h(x^k) - h(z^{k+1}))/\alpha}$ for every $k$, which is \eqref{eq:trajectory} with $\sigma(x,z) = \norm{x - z}$ and $\varphi(d) = \sqrt{d/\alpha}$. Theorem~\ref{thm:rate-transfer} with $K = N$ gives the bound. When $\rho > L_c L_F$, Lemma~\ref{lem:pl-descent} shows that \eqref{eq:pl-model-decrease} implies the quadratic decrease in the hypothesis, with
\begin{equation}
  h(x^k) - h(z^{k+1}) \geq \alpha \norm{x^k - z^{k+1}}^2, \qquad \alpha = \tfrac{1}{2}(\rho - L_c L_F). \tag{QD}\label{eq:QD}
\end{equation}
\end{proof}

\begin{assumption}[PL-BDD]\label{ass:app-pl-bdd}
The sublevel set $\{x \in \R^n : h(x) \leq h(x^0)\}$ is bounded.
\end{assumption}

With $C = \R^n$, this is the level set of Remark~\ref{rem:accumulation}, so $\{x^k\}$ is bounded and has an accumulation point.

\paragraph{Reading guide for the next theorems.}\label{rem:app-pl-purpose}
The next three theorems identify what vanishing prox-linear residual
$\norm{z^{k+1}-x^k}$ implies at accumulation points under three inner-termination templates: DC, Clarke, and Mordukhovich.
This is the same residual that appears in Lemma~\ref{lem:pl-descent} and Theorem~\ref{thm:pl-rate}, so the $O(1/\sqrt{N})$ bound is a rate statement in a concrete optimality metric.
In the augmented algorithm, accumulation points are upgraded further to almost-sure d-stationarity (Theorem~\ref{thm:random}), which is strictly stronger in general (Remark~\ref{rem:stationarity-comparison}; see also Remark~\ref{rem:pl-oracle-residual}).

\subsection{DC outer function}

\begin{assumption}[PL-DC1]\label{ass:app-pl-dc-c}
The outer function $c \colon \R^m \to \R$ is DC: $c(y)=f(y)-g(y)$ with $f,g \colon \R^m \to \R$ convex and finite-valued, and $c$ is globally Lipschitz with constant $L_c>0$.
\end{assumption}

\begin{assumption}[PL-DC2]\label{ass:app-pl-dc-F}
The inner map $F \colon \R^n \to \R^m$ is continuously differentiable and $\nabla F$ is $L_F$-Lipschitz in operator norm, as in Lemma~\ref{lem:pl-descent}.
\end{assumption}

\begin{assumption}[PL-DC3]\label{ass:app-pl-dc-rho}
The parameter satisfies $\rho > L_c L_F$.
\end{assumption}

\paragraph{Inner termination (DC criticality of the model).}
Fix $k$ and decompose $m^k = \varphi^k - \psi^k$ with
$\varphi^k(y) := f(\ell^k(y)) + \tfrac{\rho}{2}\norm{y-x^k}^2$ and $\psi^k(y) := g(\ell^k(y))$ (both convex).
The oracle output $z^{k+1}$ is \emph{DC critical} for $m^k$ if
\begin{equation}\label{eq:app-pl-dc-inner}
  \nabla F(x^k)^\top \xi^{k+1} + \rho(z^{k+1}-x^k) = \nabla F(x^k)^\top \eta^{k+1}
\end{equation}
for some $\xi^{k+1} \in \partial f(\ell^k(z^{k+1}))$ and $\eta^{k+1} \in \partial g(\ell^k(z^{k+1}))$, and satisfies~\eqref{eq:pl-model-decrease}.

\begin{theorem}[DC criticality at accumulation points]\label{thm:app-pl-dc}
Under Assumptions~\ref{ass:app-pl-dc-c}--\ref{ass:app-pl-dc-rho} and~\ref{ass:app-pl-bdd}, suppose each oracle output $z^{k+1}$ satisfies~\eqref{eq:app-pl-dc-inner} and~\eqref{eq:pl-model-decrease}.
Then every accumulation point $x^*$ of $\{x^k\}$ satisfies
\begin{equation}\label{eq:app-pl-dc-stat}
  \nabla F(x^*)^\top \partial f(F(x^*)) \;\cap\; \nabla F(x^*)^\top \partial g(F(x^*)) \;\neq\; \emptyset,
\end{equation}
i.e., DC criticality of $h=f\circ F - g\circ F$ at $x^*$.
\end{theorem}

\begin{proof}
Each $z^{k+1}$ satisfies \eqref{eq:pl-model-decrease}. Assumptions~\ref{ass:app-pl-dc-c} and~\ref{ass:app-pl-dc-F} give the Lipschitz hypotheses of Lemma~\ref{lem:pl-descent}, and $\rho > L_c L_F$ by Assumption~\ref{ass:app-pl-dc-rho}, so the lemma applies at every $k$. With greedy selection, $h(x^{k+1}) \leq h(z^{k+1})$, and therefore $h(x^k) - h(x^{k+1}) \geq \tfrac{1}{2}(\rho - L_c L_F)\,\norm{z^{k+1} - x^k}^2$ for every $k$. Since $h$ is bounded below (Section~\ref{subsec:pl-oracle}), summing over $k$ gives $\sum_k \norm{z^{k+1} - x^k}^2 < \infty$, hence $\norm{z^{k+1} - x^k} \to 0$.
Also $h(x^{k+1}) \leq h(x^k)$, so $h(x^k) \leq h(x^0)$ for all $k$ and $\{x^k\}$ lies in a compact set by Assumption~\ref{ass:app-pl-bdd}; thus $\{x^k\}$ has accumulation points.

Let $x^{k_j} \to x^*$. Since $\norm{z^{k_j+1}-x^{k_j}}\to 0$, also $z^{k_j+1} \to x^*$.
Along $k_j$, $\nabla F(x^{k_j}) \to \nabla F(x^*)$.
Set $y^j := \ell^{k_j}(z^{k_j+1}) = F(x^{k_j}) + \nabla F(x^{k_j})(z^{k_j+1}-x^{k_j}) \to F(x^*)$.
For $\xi^{k_j+1} \in \partial f(y^j)$ and $\eta^{k_j+1} \in \partial g(y^j)$, the sequences $\{\xi^{k_j+1}\}$ and $\{\eta^{k_j+1}\}$ are bounded: for all large $j$, $y^j$ lies in a fixed compact set, and convex finite $f,g$ are Lipschitz on compacts, hence their subdifferentials are uniformly bounded there.
Passing to a subsequence, $\xi^{k_j+1} \to \xi^*$ and $\eta^{k_j+1} \to \eta^*$, and outer semicontinuity of convex subdifferentials \citep[Theorem~24.4]{Rockafellar1997} yields $\xi^* \in \partial f(F(x^*))$ and $\eta^* \in \partial g(F(x^*))$.
The proximal term $\rho(z^{k_j+1}-x^{k_j}) \to 0$.
Passing to the limit in~\eqref{eq:app-pl-dc-inner} gives $\nabla F(x^*)^\top \xi^* = \nabla F(x^*)^\top \eta^*$, which is~\eqref{eq:app-pl-dc-stat}.
\end{proof}


\subsection{Clarke outer function}

\begin{assumption}[PL-C1]\label{ass:app-pl-clarke-c}
The function $c \colon \R^m \to \R$ is globally Lipschitz with constant $L_c>0$ (hence locally Lipschitz and therefore Clarke subdifferentiable).
\end{assumption}

\begin{assumption}[PL-C2]\label{ass:app-pl-clarke-F}
The map $F$ satisfies the same smoothness as in Lemma~\ref{lem:pl-descent}, and $\nabla F(x)$ is surjective (full row rank $m$) for every $x \in \R^n$.
\end{assumption}

\begin{assumption}[PL-C3]\label{ass:app-pl-clarke-rho}
$\rho > L_c L_F$, and $h = c \circ F$ is bounded below on $\R^n$ (equivalently, $h^* > -\infty$ in the notation of Section~\ref{sec:oracles}).
\end{assumption}

\paragraph{Inner termination (Clarke stationarity of the model).}
The oracle output $z^{k+1}$ satisfies
\begin{equation}\label{eq:app-pl-clarke-inner}
  0 \in \partialC\!\Bigl( c \circ \ell^k + \tfrac{\rho}{2}\norm{\,\cdot - x^k}^2 \Bigr)(z^{k+1}),
\end{equation}
together with~\eqref{eq:pl-model-decrease}.
Under surjectivity of $\nabla F(x^k)$, the Clarke chain rule \citep[Theorem~2.3.10]{clarke1990optimization} and sum rule \citep[Proposition~2.3.3]{clarke1990optimization} expand~\eqref{eq:app-pl-clarke-inner} to
\begin{equation}\label{eq:app-pl-clarke-expanded}
  0 \in \nabla F(x^k)^\top \partialC c(\ell^k(z^{k+1})) + \rho(z^{k+1}-x^k).
\end{equation}

\begin{theorem}[Clarke stationarity at accumulation points]\label{thm:app-pl-clarke}
Under Assumptions~\ref{ass:app-pl-clarke-c}--\ref{ass:app-pl-clarke-rho} and~\ref{ass:app-pl-bdd}, suppose each $z^{k+1}$ satisfies~\eqref{eq:app-pl-clarke-inner} and~\eqref{eq:pl-model-decrease}.
Then every accumulation point $x^*$ of $\{x^k\}$ satisfies $0 \in \partialC h(x^*)$, i.e., Clarke stationarity of $h=c\circ F$.
\end{theorem}

\begin{proof}
Greedy monotonicity and Assumption~\ref{ass:app-pl-bdd} imply that $\{x^k\}$ is bounded, hence it has accumulation points.
Let $x^{k_j} \to x^*$. The argument in the proof of Theorem~\ref{thm:app-pl-dc} applies under Assumptions~\ref{ass:app-pl-clarke-c}--\ref{ass:app-pl-clarke-rho} in place of Assumptions~\ref{ass:app-pl-dc-c}--\ref{ass:app-pl-dc-rho}, so $\norm{z^{k_j+1}-x^{k_j}}\to 0$, hence $z^{k_j+1} \to x^*$.
Choose $v^{k_j+1} \in \partialC c(\ell^{k_j}(z^{k_j+1}))$ with $\nabla F(x^{k_j})^\top v^{k_j+1} + \rho(z^{k_j+1}-x^{k_j})=0$.
The Clarke subdifferential is locally bounded \citep[Proposition~2.1.2(a)]{clarke1990optimization}; pass to a subsequence with $v^{k_j+1} \to v^*$, and outer semicontinuity of $\partialC c$ yields $v^* \in \partialC c(F(x^*))$.
Taking limits, $\nabla F(x^*)^\top v^* = 0$, i.e., $0 \in \nabla F(x^*)^\top \partialC c(F(x^*))$.
Under surjectivity, the Clarke chain rule for $c\circ F$ \citep[Theorem~2.3.10]{clarke1990optimization} gives $\partialC h(x^*) = \nabla F(x^*)^\top \partialC c(F(x^*))$ (equality), hence $0 \in \partialC h(x^*)$.
\end{proof}

\subsection{Mordukhovich (limiting) outer function}

\paragraph{Inner termination (Mordukhovich stationarity of the model).}
Replace~\eqref{eq:app-pl-clarke-inner} by
\begin{equation}\label{eq:app-pl-mord-inner}
  0 \in \partialL\!\Bigl( c \circ \ell^k + \tfrac{\rho}{2}\norm{\,\cdot - x^k}^2 \Bigr)(z^{k+1}),
\end{equation}
with~\eqref{eq:pl-model-decrease}, where $\partialL$ denotes the limiting (Mordukhovich) subdifferential.
Under surjectivity of $\nabla F(x^k)$, the limiting chain rule \citep[Exercise~10.7]{RockafellarWets2009} and the smooth-sum rule for $\partialL$ \citep[Exercise~10.10]{RockafellarWets2009} yield the expansion
\begin{equation}\label{eq:app-pl-mord-expanded}
  0 \in \nabla F(x^k)^\top \partialL c(\ell^k(z^{k+1})) + \rho(z^{k+1}-x^k).
\end{equation}

\begin{theorem}[Mordukhovich stationarity at accumulation points]\label{thm:app-pl-mord}
Under Assumptions~\ref{ass:app-pl-clarke-c}--\ref{ass:app-pl-clarke-rho} and~\ref{ass:app-pl-bdd}, suppose each $z^{k+1}$ satisfies~\eqref{eq:app-pl-mord-inner} and~\eqref{eq:pl-model-decrease}.
Then every accumulation point $x^*$ of $\{x^k\}$ satisfies $0 \in \partialL h(x^*)$, i.e., Mordukhovich stationarity of $h=c\circ F$.
\end{theorem}

\begin{proof}
Greedy monotonicity and Assumption~\ref{ass:app-pl-bdd} imply that $\{x^k\}$ is bounded, hence it has accumulation points.
Let $x^{k_j} \to x^*$; the argument in the proof of Theorem~\ref{thm:app-pl-dc} applies under Assumptions~\ref{ass:app-pl-clarke-c}--\ref{ass:app-pl-clarke-rho} in place of Assumptions~\ref{ass:app-pl-dc-c}--\ref{ass:app-pl-dc-rho}, so $\norm{z^{k_j+1}-x^{k_j}}\to 0$, hence $z^{k_j+1} \to x^*$.
Pick $v^{k_j+1} \in \partialL c(\ell^{k_j}(z^{k_j+1}))$ with $\nabla F(x^{k_j})^\top v^{k_j+1} + \rho(z^{k_j+1}-x^{k_j})=0$.
Since $\partialL c(y) \subseteq \partialC c(y)$ for locally Lipschitz $c$ \citep[Theorems~8.49 and~9.13]{RockafellarWets2009} and $\partialC c$ is locally bounded, the $v^{k_j+1}$ are bounded; pass to a convergent subsequence $v^{k_j+1} \to v^*$.
Outer semicontinuity of $\partialL c$ gives $v^* \in \partialL c(F(x^*))$, and taking limits yields $0 \in \nabla F(x^*)^\top \partialL c(F(x^*))$.
Under surjectivity, $\partialL h(x^*) = \nabla F(x^*)^\top \partialL c(F(x^*))$ \citep[Exercise~10.7]{RockafellarWets2009}, hence $0 \in \partialL h(x^*)$.
\end{proof}

\begin{remark}[Comparison]\label{rem:app-pl-compare}
The descent step is common (Lemma~\ref{lem:pl-descent}); only the subdifferential calculus in the limit and the final chain rule differ.
Since $\partialL c(y) \subseteq \partialC c(y)$ always, Mordukhovich stationarity of $h$ is stronger than Clarke stationarity, and~\eqref{eq:app-pl-mord-inner} is a stronger inner requirement than~\eqref{eq:app-pl-clarke-inner}.
Assumption~\ref{ass:app-pl-bdd} is imposed uniformly in Theorems~\ref{thm:app-pl-dc}--\ref{thm:app-pl-mord} so that greedy iterates admit accumulation points; it is not needed for the algebraic form of the limit along a convergent subsequence.
\end{remark}

\section{Proofs for Section~\ref{subsec:gs-oracle}: INGD}\label{sec:appendix-ingd}

Throughout this section, $C = \R^n$. Hence the subspace of Definition~\ref{def:affine-hull} is $\R^n$, $D_C(x) = \R^n$ for every $x$, and $\bar{x}$ is d-stationary if and only if $h'(\bar{x};\,d) \geq 0$ for every $d \in \R^n$. In this section, $L$ denotes the Lipschitz constant of $h$.

Let $h \in \mathcal{F}(\Delta, L)$ \citep[(2)]{Zhang2020}: $h$ is $L$-Lipschitz on $\R^n$, Hadamard directionally differentiable at every point, and $h(x^0) - \inf h \leq \Delta$. Then $h$ satisfies Assumptions~A1 and~A2, and $\Delta_0 \leq \Delta$. Every $g \in \partialC h(y)$ satisfies $\norm{g} \leq L$ \citep[Proposition~2.1.2(a)]{clarke1990optimization}. Following \citet[Assumption~1(a)]{Zhang2020}, the method accesses $h$ through a generalized-gradient oracle $\mathsf{O}$, which, given $y, d \in \R^n$, returns $h(y)$ and a vector $g \in \partialC h(y)$ with $\inner{g}{d} = h'(y;\,d)$. We assume that $\mathsf{O}$ is deterministic and that the map $(y, d) \mapsto g$ is Borel measurable; the conditional expectations in the proof of Lemma~\ref{lem:ingd-fail}, as in that of \citet[Lemma~13]{Zhang2020}, require it.

The symbols of \citet[Algorithm~1]{Zhang2020} denote other objects in this paper and are replaced as follows.
\begin{center}
\begin{tabular}{lll}
\citet{Zhang2020} & here & reason \\
\hline
$x_t$ & $x$ & one pass at a fixed anchor; $t$ denotes step sizes \\
$x_{t,k}$ & $w_j$ & $x^k$ denotes the iterates of Algorithm~\ref{alg:random} \\
$m_{t,k}$, $g_{t,k+1}$, $\beta_{t,k}$ & $m_j$, $g_{j+1}$, $\beta_j$ & $k$ is the iteration index of Algorithm~\ref{alg:random} \\
$y_{t,k+1}$ & $u_{j+1}$ & $y^{k+1}$ is the exploration candidate \\
$K$ & $J$ & $K$ is the horizon in Proposition~\ref{prop:budget} \\
$\gamma$ & $\zeta$ & $\gamma$ is the parameter of Algorithm~\ref{alg:ERP} \\
$O$ & $\mathsf{O}$ & $\mathcal{O}_C$ is the oracle of Algorithm~\ref{alg:random}
\end{tabular}
\end{center}

Fix $\delta > 0$ and $\varepsilon > 0$, set $J := \lceil 48L^2/\varepsilon^2 \rceil$, and let $\xi := (\lambda_2, \dots, \lambda_{J+1})$, where $\lambda_2, \dots, \lambda_{J+1}$ are independent and uniform on $[0,1]$. Given $x \in \R^n$, one call of the single-pass INGD oracle computes $\Psi(x, \xi) = (\mathcal{O}_C(x), \ell)$ as follows. If $L = 0$, return $(x, \mathrm{CERT})$. Otherwise:
\begin{enumerate}
  \item Set $m_1 := g$, where $(h(x), g) = \mathsf{O}(x, 0)$.
  \item For $j = 1, \dots, J$:
  \begin{enumerate}
    \item if $\norm{m_j} \leq \varepsilon$, return $(x, \mathrm{CERT})$;
    \item set $w_j := x - \delta\, m_j/\norm{m_j}$; if $h(w_j) - h(x) < -\tfrac{\delta}{4}\norm{m_j}$, return $(w_j, \mathrm{PROG})$;
    \item set $u_{j+1} := x + \lambda_{j+1}(w_j - x)$, let $g_{j+1}$ be the vector returned by $\mathsf{O}(u_{j+1}, -m_j)$, and set $m_{j+1} := \beta_j m_j + (1 - \beta_j) g_{j+1}$, where $\beta_j := (4L^2 - \norm{m_j}^2)/(4L^2 + 2\norm{m_j}^2)$.
  \end{enumerate}
  \item Return $(x, \mathrm{FAIL})$.
\end{enumerate}
These are lines 4--17 of \citet[Algorithm~1]{Zhang2020} at the anchor $x$, in the symbols above, with its three outcomes (termination, the descent test, exhaustion of the loop) labelled $\mathrm{CERT}$, $\mathrm{PROG}$ and $\mathrm{FAIL}$. There, the normalized step precedes the test of step (a); here the order is reversed, so that $w_j$ is defined, since $\norm{m_j} > \varepsilon > 0$ in step (b). The returned point and label are unchanged. If $L = 0$, the call returns before step 1 and makes no call to $\mathsf{O}$. Otherwise a call makes at most $J + 1$ calls to $\mathsf{O}$, one in step 1 and one per execution of step (c), and at most $J$ further function evaluations, at the points $w_j$.

The map $\Psi \colon \R^n \times [0,1]^J \to \R^n \times \{\mathrm{CERT}, \mathrm{PROG}, \mathrm{FAIL}\}$ is jointly Borel measurable: it is composed of $h$, which is continuous, of $\mathsf{O}$ and of arithmetic operations, through finitely many branches on Borel sets. In Algorithm~\ref{alg:random}, the call at iteration $k$ computes $(z^{k+1}, \ell^k) = \Psi(x^k, \xi^k)$, where $\xi^0, \xi^1, \dots$ are independent copies of $\xi$, independent of the exploration samples. Thus (O1) holds. Since $x^0$ is fixed and $\mathcal{G}_{k-1}$ is generated by functions of $(v^i_{\mathrm{ex}}, \hat t^i, \xi^i)_{i<k}$, the sample $\xi^k$ is independent of $\mathcal{G}_{k-1}$. For the same reason, $(v^k_{\mathrm{ex}}, \hat t^k)$ is independent of the $\sigma$-algebra $\mathcal{F}_{k-1}$ of Remark~\ref{rem:sampling}, as the end of Section~\ref{sec:convergence} requires.

\begin{lemma}[One call at a fixed anchor; adapted from the proof of {\citet[Theorem~8]{Zhang2020}}]\label{lem:ingd-call}
Let $x \in \R^n$. For every realization of $\xi$, the following hold.
\begin{enumerate}
  \item[(i)] Every vector $m_j$ computed by the call satisfies $\norm{m_j} \leq L$ and $m_j \in \partial_\delta h(x)$, and every $\beta_j$ computed lies in $[1/2, 1]$.
  \item[(ii)] If $\ell \in \{\mathrm{CERT}, \mathrm{FAIL}\}$, then $\mathcal{O}_C(x) = x$. If $\ell = \mathrm{PROG}$, then $h(x) - h(\mathcal{O}_C(x)) > \delta\varepsilon/4$.
  \item[(iii)] If $\ell = \mathrm{CERT}$, then $x$ is $(\delta, \varepsilon)$-Goldstein stationary.
\end{enumerate}
\end{lemma}

\begin{proof}
(i) If $L = 0$, the call returns before any $m_j$ or $\beta_j$ is computed. Now assume $L > 0$. We prove part (i) by induction on $j$. Since $m_1 \in \partialC h(x)$ and $x \in B[x, \delta]$, we have $m_1 \in \partial_\delta h(x)$ and $\norm{m_1} \leq L$. Suppose that $m_j$ satisfies (i) and that step (c) is executed at $j$. The function $s \mapsto (4L^2 - s)/(4L^2 + 2s)$ decreases from $1$ at $s = 0$ to $1/2$ at $s = L^2$, so $\beta_j \in [1/2, 1]$. Since $\norm{w_j - x} = \delta$, the point $u_{j+1}$ lies on the segment $[x, w_j] \subseteq B[x, \delta]$, so $g_{j+1} \in \partialC h(u_{j+1}) \subseteq \partial_\delta h(x)$ and $\norm{g_{j+1}} \leq L$. Hence $m_{j+1}$, a convex combination of $m_j$ and $g_{j+1}$, lies in the convex set $\partial_\delta h(x)$ and satisfies $\norm{m_{j+1}} \leq L$.

(ii) The first claim holds by construction. If the call returns $\mathrm{PROG}$ at step $j$, then step (a) did not return at $j$, so $\norm{m_j} > \varepsilon$, and step (b) gives $h(w_j) - h(x) < -\tfrac{\delta}{4}\norm{m_j} < -\delta\varepsilon/4$.

(iii) If $L = 0$, then $h$ is constant, so $0 \in \partialC h(x) \subseteq \partial_\delta h(x)$ and $\mathrm{dist}(0, \partial_\delta h(x)) = 0 \leq \varepsilon$. If the call returns $\mathrm{CERT}$ at step $j$, then $\norm{m_j} \leq \varepsilon$, and $m_j \in \partial_\delta h(x)$ by (i), so $\mathrm{dist}(0, \partial_\delta h(x)) \leq \varepsilon$.
\end{proof}

\begin{lemma}[Per-call failure bound; adapted from {\citet[Lemma~13]{Zhang2020}}]\label{lem:ingd-fail}
For every $x \in \R^n$, the probability that $\Psi(x, \xi)$ has label $\mathrm{FAIL}$ is at most $1/16$.
\end{lemma}

\begin{proof}
If $L = 0$, the call returns $\mathrm{CERT}$, so the probability is $0$. Assume $L > 0$. Fix $x$. For the proof only, set $m_{j'} := 0$ for every $j' > j_0$ if the call returns at step $j_0$, so that $m_1, \dots, m_J$ are defined on every realization. Let $\mathcal{H}_1$ be the trivial $\sigma$-algebra and $\mathcal{H}_j := \sigma(\lambda_2, \dots, \lambda_j)$ for $j \geq 2$. Then $m_j$ and $w_j$ are $\mathcal{H}_j$-measurable, and $\lambda_{j+1}$ is independent of $\mathcal{H}_j$. Let $D_j \in \mathcal{H}_j$ be the event that step (c) is executed at $j$. On $D_j$, $\norm{m_j} > \varepsilon$ and $h(w_j) - h(x) \geq -\tfrac{\delta}{4}\norm{m_j}$.

On $D_j$, $w_j - x$ is a positive multiple of $-m_j$, so the defining property of $\mathsf{O}$ and positive homogeneity of $h'(u_{j+1};\,\cdot)$ give $\inner{g_{j+1}}{w_j - x} = h'(u_{j+1};\, w_j - x)$. Since $u_{j+1}$ is uniform on $[x, w_j]$ given $\mathcal{H}_j$, \citet[Lemma~3]{Zhang2020} yields, on $D_j$,
\begin{equation*}
  \mathbb{E}\big[\inner{g_{j+1}}{w_j - x} \mid \mathcal{H}_j\big] = \int_0^1 h'\big(x + \lambda(w_j - x);\, w_j - x\big)\, d\lambda = h(w_j) - h(x).
\end{equation*}
Since $m_j = -(\norm{m_j}/\delta)(w_j - x)$, it follows that, on $D_j$,
\begin{equation*}
  \mathbb{E}\big[\inner{g_{j+1}}{m_j} \mid \mathcal{H}_j\big] = -\frac{\norm{m_j}}{\delta}\big(h(w_j) - h(x)\big) \leq \frac{\norm{m_j}^2}{4}.
\end{equation*}
Using $\norm{g_{j+1}} \leq L$, and that $\beta_j \in [0,1]$ is $\mathcal{H}_j$-measurable, we obtain on $D_j$
\begin{equation*}
  \mathbb{E}\big[\norm{m_{j+1}}^2 \mid \mathcal{H}_j\big] \leq \chi(\beta_j),
\end{equation*}
where $a := \norm{m_j}^2$ and $\chi(\beta) := \beta^2 a + (1 - \beta)^2 L^2 + \tfrac12 \beta (1 - \beta) a$. Expanding,
\begin{equation*}
  \chi(\beta) = \beta^2 \big(L^2 + \tfrac{a}{2}\big) + \beta\big(\tfrac{a}{2} - 2L^2\big) + L^2.
\end{equation*}
The derivative $2\beta(L^2 + a/2) + a/2 - 2L^2$ vanishes at $\beta = (4L^2 - a)/(4L^2 + 2a)$, the value of step (c), and the second derivative $2L^2 + a$ is positive. The vertex value is
\begin{equation*}
  \chi(\beta_j) = L^2 - \frac{(4L^2 - a)^2}{16(L^2 + a/2)} = \frac{L^2 a - a^2/16}{L^2 + a/2}.
\end{equation*}
For $a \in [0, L^2]$,
\begin{equation*}
  \frac{L^2 a - a^2/16}{L^2 + a/2} \leq \Big(1 - \frac{a}{3L^2}\Big) a,
\end{equation*}
since for $a > 0$ it rearranges to $a/6 + a/16 \geq a^2/(6L^2)$, and multiplying through by $48L^2/a$ gives $11L^2 \geq 8a$. Thus $\chi(\beta_j) \leq \theta(a)$, where $\theta(s) := (1 - s/(3L^2))\, s$; the comparison uses $\norm{m_j} \leq L$ (Lemma~\ref{lem:ingd-call}). Off $D_j$, $m_{j+1} = 0$. Since $\theta \geq 0$ on $[0, L^2]$, it follows that $\mathbb{E}[\norm{m_{j+1}}^2 \mid \mathcal{H}_j] \leq \theta(\norm{m_j}^2)$ on every realization.

Let $v_j := \mathbb{E}[\norm{m_j}^2]/L^2 \in [0,1]$. Taking expectations and applying Jensen's inequality to the concave function $\theta$ gives $v_{j+1} \leq v_j - v_j^2/3$. Since $s \mapsto s - s^2/3$ is increasing on $[0, 1]$, induction from $v_1 \leq 1$ gives $v_j \leq 3/(j+2)$: if $v_j \leq 3/(j+2)$, then $v_{j+1} \leq 3/(j+2) - 3/(j+2)^2 = 3(j+1)/(j+2)^2 \leq 3/(j+3)$. In particular, $\mathbb{E}[\norm{m_J}^2] \leq 3L^2/(J+2) \leq \varepsilon^2/16$, since $J \geq 48L^2/\varepsilon^2$. If the label is $\mathrm{FAIL}$, step (a) did not return at $j = J$, so $\norm{m_J} > \varepsilon$. By Markov's inequality, the label is $\mathrm{FAIL}$ with probability at most $\mathbb{E}[\norm{m_J}^2]/\varepsilon^2 \leq 1/16$.
\end{proof}

\begin{proof}[Proof of Theorem~\ref{thm:ingd}]
Set $T := \lceil 4\Delta_0/(\delta\varepsilon)\rceil$ and $N := T + \max\{3T, \lceil 11 \log(1/\zeta)\rceil\}$. We show that, with probability at least $1 - \zeta$, some call $\mathcal{O}_C(x^k)$ with $k < N$ returns $\mathrm{CERT}$, and that the calls $\mathcal{O}_C(x^0)$, \ldots, $\mathcal{O}_C(x^{N-1})$ make at most $N(J + 1)$ calls to $\mathsf{O}$.
Condition (O2) holds since $C = \R^n$, and (O1) was shown above. By Lemma~\ref{lem:ingd-call}(ii), (O3) and (P2) hold for every realization, with $\bar\eta = \delta\varepsilon/4$. By Lemma~\ref{lem:ingd-call}(iii), (P1) holds with $(\delta,\varepsilon)$-Goldstein stationarity as the target criterion. By Lemma~\ref{lem:ingd-fail}, the call at any fixed $x$ has label $\mathrm{FAIL}$ with probability at most $1/16$. Since $\xi^k$ is independent of $\mathcal{G}_{k-1}$, $x^k$ is $\mathcal{G}_{k-1}$-measurable, and $\Psi$ is jointly measurable, Fubini's theorem gives $\mathbb{P}(\ell^k = \mathrm{FAIL} \mid \mathcal{G}_{k-1}) \leq 1/16 \leq 1/4$ almost surely, which is (P3) with $q = 1/4 < 1/e$. Theorem~\ref{thm:certificate} therefore applies. Its proof in Appendix~\ref{sec:appendix-rate-transfer} admits any $c > eq/(1-eq) = e/(4-e)$; we take $c = 3$, for which $\rho = e/3$. Since $\lceil \Delta_0/\bar\eta \rceil = T$ and $1/\log(3/e) < 11$, the horizon of that proof is at most $N$, and a longer horizon only enlarges the event that some call returns $\mathrm{CERT}$. This proves the claim on $\mathrm{CERT}$. Each call makes at most $J+1$ calls to $\mathsf{O}$. If $\varepsilon \leq L$, then $J + 1 \leq 48L^2/\varepsilon^2 + 2 \leq 50L^2/\varepsilon^2$. Moreover, $\lceil a \rceil \leq a + 1$ and $\max\{a, b\} \leq a + b$ for $a, b \geq 0$ give $N \leq 4T + 11\log(1/\zeta) + 1 \leq 16\Delta_0/(\delta\varepsilon) + 11\log(1/\zeta) + 5$. Multiplying the two bounds gives $N(J+1) \leq 800\, L^2 \varepsilon^{-2} \big(1 + \Delta_0 \delta^{-1} \varepsilon^{-1} + \log(1/\zeta)\big)$. Since the certifying call has index $k < N$, the calls up to and including it are among $\mathcal{O}_C(x^0), \dots, \mathcal{O}_C(x^{N-1})$, which proves the bound of the theorem.
\end{proof}

\begin{remark}\label{rem:ingd-large-eps}
If $\varepsilon > L$, every point is $(\delta,\varepsilon)$-Goldstein stationary, since $\partialC h(x) \subseteq \partial_\delta h(x)$ contains a vector of norm at most $L$. If $L = 0$, the call returns $\mathrm{CERT}$ before step 1. If $L > 0$, then $J \geq 1$ and, by Lemma~\ref{lem:ingd-call}(i), $\norm{m_1} \leq L < \varepsilon$, so every call returns $\mathrm{CERT}$ at $j = 1$. The restriction $\varepsilon \leq L$ in Theorem~\ref{thm:ingd} excludes only this case.
\end{remark}

\begin{remark}\label{rem:ingd-base}
Proposition~\ref{prop:budget} holds verbatim for the base method without exploration (Section~\ref{sec:rate-transfer}), and the proof of Theorem~\ref{thm:certificate} uses Algorithm~\ref{alg:random} only through Proposition~\ref{prop:budget}, (O3) and (P1)--(P3). Up to its first certificate, the base method with the single-pass INGD oracle performs the steps of INGD \citep[Algorithm~1]{Zhang2020} with its outer loop left unbounded: an exhausted pass is restarted at the same anchor with fresh randomness, and an accepted step moves the anchor to $w_j$. Hence the bound of Theorem~\ref{thm:ingd} also holds for INGD in this form, with $\log(1/\zeta)$ entering additively rather than through the factor $\log(4\Delta/(\zeta\delta\varepsilon))$ of \citet[Theorem~8]{Zhang2020}.
\end{remark}

\section{A finite-time bound for the exploration step}\label{sec:appendix-rep-finite}

This section makes two assumptions beyond Assumptions~A1 and~A2: $C = \R^n$, and $h$ is Lipschitz on $\R^n$, not only locally Lipschitz. In this section, $L$ denotes the Lipschitz constant of $h$, and we assume $L > 0$; if $L = 0$, then $h$ is constant and every point is $(\delta,\varepsilon)$-Goldstein stationary. Since $C = \R^n$, the subspace of Definition~\ref{def:affine-hull} is $\R^n$, and every oracle is feasible. Assumption~A1 gives boundedness below: $h^* := \inf_{x \in C} h(x)$ is finite, and $\Delta_0 := h(x^0) - h^*$, where $x^0$ is the initial point of Algorithm~\ref{alg:random}. Assumption~A2 is not used. The Goldstein $\delta$-subdifferential $\partial_\delta h$ and $(\delta,\varepsilon)$-Goldstein stationarity are as in Definition~\ref{def:goldstein}, and $\mathcal{S}_{\delta,\varepsilon}$ denotes the set of $(\delta,\varepsilon)$-Goldstein stationary points.

At a point $x$, the methods of \citet{Zhang2020} and \citet{Davis2022} seek a vector of $\partial_\delta h(x)$ that either certifies that $x$ is $(\delta,\varepsilon)$-Goldstein stationary or yields a decrease of order $\delta\varepsilon$ along a step of length $\delta$. The exploration step uses function values only and certifies nothing, but at every point outside $\mathcal{S}_{\delta,\varepsilon}$ it samples such a step with a probability bounded below by a quantity that is exponentially small in $n$.

\begin{lemma}\label{lem:goldstein-set}
For every $x \in \R^n$, the set $\partial_\delta h(x)$ is nonempty, convex and compact, and it is contained in $B[0, L]$. If $x_i \to x$, $w_i \in \partial_\delta h(x_i)$ and $w_i \to w$, then $w \in \partial_\delta h(x)$. Consequently, $\mathcal{S}_{\delta,\varepsilon}$ is closed.
\end{lemma}

\begin{proof}
For every $y$, the set $\partialC h(y)$ is nonempty and every $g \in \partialC h(y)$ satisfies $\norm{g} \leq L$ \citep[Proposition~2.1.2(a)]{clarke1990optimization}, so $\partial_\delta h(x)$ is a nonempty convex subset of $B[0, L]$. By Carath\'eodory's theorem, every $w \in \partial_\delta h(x)$ can be written as $w = \sum_{j=0}^{n} \lambda_j g_j$, with $(\lambda_j)_{j=0}^{n}$ in the unit simplex of $\R^{n+1}$, $g_j \in \partialC h(y_j)$ and $y_j \in B[x, \delta]$. Let $x_i$, $w_i$ and $w$ be as in the statement, and write each $w_i = \sum_{j=0}^{n} \lambda_{ij} g_{ij}$, with $(\lambda_{ij})_{j=0}^{n}$ in the unit simplex, $g_{ij} \in \partialC h(y_{ij})$ and $y_{ij} \in B[x_i, \delta]$. These sequences are bounded, so along a subsequence $\lambda_{ij} \to \lambda_j$, $y_{ij} \to y_j$ and $g_{ij} \to g_j$ for every $j$. Then $y_j \in B[x, \delta]$, $g_j \in \partialC h(y_j)$ \citep[Proposition~2.1.5(b)]{clarke1990optimization}, and $w = \sum_j \lambda_j g_j \in \partial_\delta h(x)$. Taking $x_i = x$ for every $i$ shows that $\partial_\delta h(x)$ is closed, hence compact. Finally, let $x_i \in \mathcal{S}_{\delta,\varepsilon}$ with $x_i \to x$. By compactness, there are $w_i \in \partial_\delta h(x_i)$ with $\norm{w_i} \leq \varepsilon$. Along a subsequence, $w_i \to w$, with $\norm{w} \leq \varepsilon$ and $w \in \partial_\delta h(x)$, so $x \in \mathcal{S}_{\delta,\varepsilon}$.
\end{proof}

\begin{lemma}\label{lem:goldstein-descent}
Let $x \notin \mathcal{S}_{\delta,\varepsilon}$, let $\bar g$ be the element of minimal norm of $\partial_\delta h(x)$, and let $\hat g := \bar g/\norm{\bar g}$. Then $\norm{\bar g} > \varepsilon$, and the following hold.
\begin{enumerate}
  \item[(i)] For every $t \in [0, \delta]$, $h(x - t\hat g) \leq h(x) - t\norm{\bar g}$.
  \item[(ii)] For every $v \in S^{n-1}$ with $\norm{v + \hat g} \leq \varepsilon/(2L)$ and every $t \in (0, \delta]$, $h(x + tv) < h(x) - t\varepsilon/2$.
\end{enumerate}
\end{lemma}

\begin{proof}
By Lemma~\ref{lem:goldstein-set}, $\bar g$ exists and is unique, and $\norm{\bar g} = \mathrm{dist}(0, \partial_\delta h(x)) > \varepsilon$. Since $\bar g$ is the projection of $0$ onto the closed convex set $\partial_\delta h(x)$, every $w \in \partial_\delta h(x)$ satisfies $\inner{w - \bar g}{\bar g} \geq 0$, that is, $\inner{w}{\hat g} \geq \norm{\bar g}$.

(i) The case $t = 0$ is trivial. Let $t \in (0, \delta]$. By Lebourg's mean value theorem \citep[Theorem~2.3.7]{clarke1990optimization}, there are a point $y$ on the open segment between $x$ and $x - t\hat g$ and a vector $w \in \partialC h(y)$ with $h(x - t\hat g) - h(x) = -t\inner{w}{\hat g}$. Since $\norm{y - x} < t \leq \delta$, $w \in \partial_\delta h(x)$, and therefore $h(x - t\hat g) - h(x) \leq -t\norm{\bar g}$.

(ii) By $L$-Lipschitz continuity of $h$, (i) and $\norm{\bar g} > \varepsilon$,
\begin{equation*}
  h(x + tv) \leq h(x - t\hat g) + Lt\norm{v + \hat g} \leq h(x) - t\norm{\bar g} + \frac{t\varepsilon}{2} < h(x) - \frac{t\varepsilon}{2}. \qedhere
\end{equation*}
\end{proof}

For $t = \delta$, the bound in (i) is \citet[Lemma~3.1]{Goldstein1977} applied to the direction $-\hat g$, with Goldstein's $q$ equal to $\norm{\bar g}$.

\begin{lemma}\label{lem:cap}
Let $v$ be uniformly distributed on $S^{n-1}$, and let $s > 0$. The probability $\pi_n(s) := \mathbb{P}(\norm{v - u} \leq s)$ is the same for every $u \in S^{n-1}$, and $\pi_n(s) \geq \big(s/(s+2)\big)^n$.
\end{lemma}

\begin{proof}
The uniform distribution on $S^{n-1}$ is invariant under orthogonal maps, which preserve distances and act transitively on $S^{n-1}$, so $\pi_n(s)$ does not depend on $u$. Let $u_1, \ldots, u_M \in S^{n-1}$ satisfy $\norm{u_i - u_j} > s$ for $i \neq j$. The open balls of radius $s/2$ centred at the $u_i$ are pairwise disjoint and contained in the open ball of radius $1 + s/2$ centred at $0$, so comparing volumes gives $M(s/2)^n \leq (1 + s/2)^n$, that is, $M \leq (1 + 2/s)^n$. Choose such points with $M$ maximal. Then every $v \in S^{n-1}$ satisfies $\norm{v - u_i} \leq s$ for some $i$, so $1 \leq \sum_{i=1}^{M} \mathbb{P}(\norm{v - u_i} \leq s) = M\pi_n(s)$, and $\pi_n(s) \geq (1 + 2/s)^{-n} = \big(s/(s+2)\big)^n$.
\end{proof}

\begin{proposition}\label{prop:rep-finite}
Let $C = \R^n$, let $h$ be $L$-Lipschitz on $\R^n$ with $L > 0$ and bounded below, and let $\{x^k\}$ be generated by Algorithm~\ref{alg:random} with any oracle $\mathcal{O}_C$, randomized or adaptive, where in Algorithm~\ref{alg:ERP} the direction $v^k_{\mathrm{ex}}$ is uniform on $S^{n-1}$, the step size $\hat t^k$ is uniform on $[0, r]$, and $v^k_{\mathrm{ex}}$ and $\hat t^k$ are independent of each other and of the $\sigma$-algebra $\mathcal{F}_{k-1}$ of Remark~\ref{rem:sampling}. Let $\delta := r$ and $\varepsilon > \gamma r$, and let $Z := \#\{k \geq 0 : x^k \notin \mathcal{S}_{\delta,\varepsilon}\}$ be the number of iterates that are not $(\delta,\varepsilon)$-Goldstein stationary. Then
\begin{equation*}
  \mathbb{E}[Z] \leq \frac{8\Delta_0}{\delta\varepsilon\, \pi_n\big(\varepsilon/(2L)\big)} \leq \frac{8\Delta_0}{\delta\varepsilon}\Big(1 + \frac{4L}{\varepsilon}\Big)^{n}.
\end{equation*}
In particular, almost surely all but finitely many iterates are $(\delta,\varepsilon)$-Goldstein stationary, and the index $\nu := \inf\{k \geq 0 : x^k \in \mathcal{S}_{\delta,\varepsilon}\}$ of the first such iterate satisfies $\mathbb{E}[\nu] \leq \mathbb{E}[Z]$.
\end{proposition}

These distributions satisfy Remark~\ref{rem:sampling}, so Theorem~\ref{thm:random} applies to the same run when Assumption~A2 holds.

\begin{proof}
By greedy selection, $h(x^{k+1}) \leq h(y^{k+1})$, and $h(y^{k+1}) \leq h(x^k)$ by the definition of $t^k_{ex}$ in Algorithm~\ref{alg:ERP}, for every $k$. Let $A$ be the set of $(x, v, t) \in \R^n \times S^{n-1} \times [0, r]$ with $h(x + tv) + \frac{\gamma}{2}t^2 < h(x)$ and $h(x + tv) \leq h(x) - \delta\varepsilon/4$. It is a Borel set, since $h$ is continuous. Let $A_k := \{(x^k, v^k_{\mathrm{ex}}, \hat t^k) \in A\}$. On $A_k$, the exploration step is accepted, since $C = \R^n$, so $y^{k+1} = x^k + \hat t^k v^k_{\mathrm{ex}}$ and $h(x^{k+1}) \leq h(y^{k+1}) \leq h(x^k) - \delta\varepsilon/4$. Since $h(x^k)$ is nonincreasing and $h \geq h^*$, on every realization,
\begin{equation*}
  \sum_{k \geq 0} \mathbf{1}_{A_k} \leq \frac{4\Delta_0}{\delta\varepsilon}.
\end{equation*}

Let $x \notin \mathcal{S}_{\delta,\varepsilon}$, and let $\hat g$ be as in Lemma~\ref{lem:goldstein-descent}. If $v \in S^{n-1}$ satisfies $\norm{v + \hat g} \leq \varepsilon/(2L)$ and $t \in [\delta/2, \delta]$, then Lemma~\ref{lem:goldstein-descent}(ii) gives $h(x + tv) < h(x) - t\varepsilon/2 \leq h(x) - \delta\varepsilon/4$, and, since $\gamma t \leq \gamma r < \varepsilon$,
\begin{equation*}
  h(x + tv) + \frac{\gamma}{2}t^2 < h(x) - \frac{t}{2}(\varepsilon - \gamma t) < h(x),
\end{equation*}
so $(x, v, t) \in A$. Let $v$ be uniform on $S^{n-1}$ and $t$ uniform on $[0, r] = [0, \delta]$, independent. By Lemma~\ref{lem:cap} with $u = -\hat g$, $a(x) := \mathbb{P}\big((x, v, t) \in A\big) \geq \frac12 \pi_n\big(\varepsilon/(2L)\big) =: p$.

The pair $(v^k_{\mathrm{ex}}, \hat t^k)$ is independent of $\mathcal{F}_{k-1}$ and $x^k$ is $\mathcal{F}_{k-1}$-measurable, so by Fubini's theorem $a$ is measurable and $\mathbb{P}(A_k \mid \mathcal{F}_{k-1}) = a(x^k)$ almost surely. The event $\{x^k \notin \mathcal{S}_{\delta,\varepsilon}\}$ belongs to $\mathcal{F}_{k-1}$, because $\mathcal{S}_{\delta,\varepsilon}$ is closed (Lemma~\ref{lem:goldstein-set}), and $a(x^k) \geq p$ on it. Hence, for every $N \geq 1$,
\begin{equation*}
  p\, \mathbb{E}\Big[\sum_{k<N} \mathbf{1}\{x^k \notin \mathcal{S}_{\delta,\varepsilon}\}\Big] \leq \sum_{k<N} \mathbb{E}\big[\mathbf{1}\{x^k \notin \mathcal{S}_{\delta,\varepsilon}\}\, a(x^k)\big] = \sum_{k<N} \mathbb{E}\big[\mathbf{1}\{x^k \notin \mathcal{S}_{\delta,\varepsilon}\}\, \mathbf{1}_{A_k}\big] \leq \frac{4\Delta_0}{\delta\varepsilon}.
\end{equation*}
Letting $N \to \infty$, monotone convergence gives $\mathbb{E}[Z] \leq 4\Delta_0/(p\,\delta\varepsilon)$, which is the first bound. The second follows from Lemma~\ref{lem:cap} with $s = \varepsilon/(2L)$, for which $(s + 2)/s = 1 + 4L/\varepsilon$. Since $\mathbb{E}[Z] < \infty$, $Z$ is finite almost surely. Finally, $x^0, \ldots, x^{\nu-1} \notin \mathcal{S}_{\delta,\varepsilon}$, so $\nu \leq Z$.
\end{proof}

\begin{remark}\label{rem:rep-scope}
Proposition~\ref{prop:rep-finite} bounds, in expectation, the number of iterates that are not $(\delta,\varepsilon)$-Goldstein stationary, and hence the index of the first one that is. Algorithm~\ref{alg:random} does not identify the stationary iterates: unlike the INGD oracle of Section~\ref{subsec:gs-oracle}, the exploration step computes no element of $\partial_\delta h$ and returns no certificate. The bound uses no property of the oracle. For $\mathcal{O}_C(x) = x$, each iteration makes at most one function evaluation (Remark~\ref{rem:cost}), so the bound on $\mathbb{E}[\nu]$ also bounds the expected number of function evaluations made before the first $(\delta,\varepsilon)$-Goldstein stationary iterate, apart from that of $h(x^0)$. The target is $(\delta,\varepsilon)$-Goldstein stationarity, not d-stationarity, and the bound requires $C = \R^n$ and a global Lipschitz constant. We do not establish a finite-iteration bound under Assumptions~A1 and~A2 alone, or for a constrained set $C$.
\end{remark}

\begin{remark}\label{rem:rep-sampler}
The proof of Proposition~\ref{prop:rep-finite} uses the uniform distribution of $v^k_{\mathrm{ex}}$ only through the bound $\mathbb{P}(\norm{v^k_{\mathrm{ex}} - u} \leq s) \geq \pi_n(s)$ for every $u \in S^{n-1}$, with $s = \varepsilon/(2L)$. For any other direction distribution $\mathcal{D}_v$ on $S^{n-1}$, fixed over the iterations, the same proof gives the first bound with $\pi_n(s)$ replaced by $\inf_{u \in S^{n-1}} \mathcal{D}_v(\{v : \norm{v - u} \leq s\})$ whenever this infimum is positive. If the infimum is zero, that argument gives no finite bound. This infimum is at most the average of $\mathcal{D}_v(\{v : \norm{v - u} \leq s\})$ over $u$ uniform on $S^{n-1}$, which equals $\pi_n(s)$ by Fubini's theorem and Lemma~\ref{lem:cap}. Hence no fixed direction distribution improves the bound that this argument gives. We do not claim that the exponential dependence on $n$ is necessary. Direction distributions that vary with $k$ and adapt to the local geometry of $h$ are left for future work.
\end{remark}


\section{Additional experimental details and discussion}
\label{sec:appendix-experiments-v2}

The purpose of this appendix is to document, in one place, when the exploration step
changes the outcome and whether those changes are favorable. Across all experiments, we
compare a deterministic base oracle to its augmented
version (oracle step plus one randomized exploration step), using the same initialization
per instance and three augmentation seeds. We report win/tie/loss against the base
solution using the signed gap
\[
  \Delta \;=\; h_{\mathrm{base}} - \mathrm{median}_s\, h_{\mathrm{aug},s},
\]
with tolerance $10^{-12}$. Positive $\Delta$ means augmentation is better. In addition to
mean and median gap (wins), we report McNemar's exact one-sided $p$-value on discordant
pairs (wins vs.\ losses), which directly tests the directional claim
``augmentation helps more often than it hurts.'' This choice is deliberate: our data
contain many ties in regimes where the base oracle is already strong, and tied-heavy
settings can make signed-rank tests less aligned with this directional question.

\paragraph{Shared mechanics.}
All experiments use $C=\mathbb{R}^n$, so no projection is needed. The exploration step
uses $\gamma=1$, $r=1$, accepts a candidate $x+\hat t v$ when
\[
  h(x+\hat t v)+\tfrac{\gamma}{2}\hat t^2 < h(x),
\]
and then greedily keeps the better of the oracle point and accepted exploration point.
The step size is sampled as $\hat t \sim \mathrm{Unif}([0,r])$ in all reported experiments.
The sphere sampler uses $v\sim \mathrm{Unif}(\mathbb{S}^{n-1})$. The gauss-axis sampler
first chooses a coordinate index $i$ uniformly, then samples
$g\sim \mathcal{N}(0, I + (\mu^2-1)e_i e_i^\top)$, and sets $v=g/\|g\|$; at $\mu=1$ this
reduces to sphere sampling, while larger $\mu$ concentrates samples near coordinate axes.
Whenever both samplers are run, we state explicitly which sampler is reported in the
aggregate table.

\subsection{D-stationarity diagnostics}
\label{subsec:exp-dstat-checks}

The counts ``\# non-d-stat'' in \Cref{tab:experiments-story-summary} flag terminal
iterates that fail a numerical stationarity test run once per random instance (on
the final iterate of each run). The three benchmarks use different objectives, so
the stationarity checks differ; each is defined in the following paragraphs.

\paragraph{Trimmed lasso (exact DC inclusion test).}
We use the standard DC decomposition
$h(x)=\tfrac12\|Ax-b\|^2+\lambda\|x\|_1-\lambda\,\mathrm{top}_k(|x|)=f(x)-g(x)$ with
$f,g$ convex. For such a model, $x$ is d-stationary if and only if
$\partial g(x)\subseteq \partial f(x)$ \citep{deOliveira2020}, where $\partial f(x)$ and
$\partial g(x)$ denote the convex subdifferentials of $f$ and $g$ at $x$. This inclusion is
equivalent to a single nonnegative scalar being zero: writing
$\mathrm{dist}_2(u,S)\mathrel{:=}\inf_{v\in S}\|u-v\|_2$ for nonempty closed convex
$S\subseteq\R^d$ and $u\in\R^d$ (in particular $d=1$ for each coordinate projection below),
\begin{equation}\label{eq:tl-dstat-dist-sup}
  \partial g(x)\subseteq \partial f(x)
  \quad\Longleftrightarrow\quad
  \sup_{\xi\in\partial g(x)} \mathrm{dist}_2\!\bigl(\xi,\,\partial f(x)\bigr)=0,
\end{equation}
where the supremum is taken over all subgradients $\xi\in\partial g(x)$ (the set
$\partial g(x)$ is compact for this $g$). The diagnostic below evaluates the right-hand
quantity in~\eqref{eq:tl-dstat-dist-sup} in closed form: every $\xi\in\partial g(x)$ is
determined by a top-$k$ support pattern for the magnitudes $|x_i|$, and the distance to
$\partial f(x)$ decouples across coordinates because $\partial f(x)$ is a Cartesian product
of the one-dimensional sets in~\eqref{eq:tl-partial-f}.

Write the residual gradient and the $k$th largest magnitude threshold as
\begin{equation*}
  r \mathrel{:=} A^\top(Ax-b)\in\R^n,
  \qquad
  \tau \mathrel{:=} |x|_{(k)},
\end{equation*}
where $|x|_{(k)}$ denotes the $k$th order statistic of $(|x_1|,\ldots,|x_n|)$ (so
$\mathrm{top}_k(|x|)=\sum_{j=1}^k |x|_{(j)}$). The convex subdifferential of
$f(x)=\tfrac12\|Ax-b\|^2+\lambda\|x\|_1$ is a Cartesian product of one-dimensional
sets; equivalently, its $i$th coordinate projection is
\begin{equation}\label{eq:tl-partial-f}
  (\partial f(x))_i \;=\;
  \begin{cases}
    \{\, r_i + \lambda\,\mathrm{sgn}(x_i)\,\}, & x_i \neq 0,\\[2pt]
    \bigl[\, r_i-\lambda,\; r_i+\lambda \,\bigr], & x_i = 0.
  \end{cases}
\end{equation}
Fix a tie tolerance $\delta_{\mathrm{tie}}\mathrel{:=}10^{-10}$ and define the
top-$k$ magnitude partition
\begin{equation*}
\begin{aligned}
  S_{\mathrm{high}} &\mathrel{:=} \{\, i : |x_i| > \tau + \delta_{\mathrm{tie}} \,\},\\
  S_{\mathrm{tie}} &\mathrel{:=} \{\, i : \bigl| |x_i|-\tau \bigr| \le \delta_{\mathrm{tie}} \,\},\\
  S_{\mathrm{low}} &\mathrel{:=} \{\, i : |x_i| < \tau - \delta_{\mathrm{tie}} \,\}.
\end{aligned}
\end{equation*}
When $|S_{\mathrm{tie}}|=0$ (equivalently, the top-$k$ magnitude set is uniquely determined
up to the tolerance $\delta_{\mathrm{tie}}$, so there is no nontrivial ambiguity at the
cutoff $\tau$), $\partial g(x)$ is a singleton $\{\xi\}$. The unique $\xi\in\partial g(x)$ has coordinates
$\xi_i=\lambda\,\mathrm{sgn}(x_i)$ if $i$ lies among the $k$ largest magnitudes and
$\xi_i=0$ otherwise. For this $\xi$, the coordinate projections of $\partial f(x)$ in~\eqref{eq:tl-partial-f}
yield $\delta_i \mathrel{:=} \mathrm{dist}_2\!\bigl(\xi_i,(\partial f(x))_i\bigr)$, and
\begin{equation*}
  \sup_{\eta\in\partial g(x)} \mathrm{dist}_2\!\bigl(\eta,\,\partial f(x)\bigr)
  \;=\; \mathrm{dist}_2\!\bigl(\xi,\,\partial f(x)\bigr)
  \;=\; \sqrt{\sum_{i=1}^n \delta_i^2}
  \;=:\; \mathrm{gap}(x),
\end{equation*}
the second equality being the product geometry of~\eqref{eq:tl-partial-f}. We declare
\emph{pass} if $\delta_i\le \varepsilon$ for every $i$, with
\begin{equation*}
  \varepsilon \mathrel{:=} 10^{-6}.
\end{equation*}

When $|S_{\mathrm{tie}}|>0$ and the top-$k$ support is genuinely ambiguous (equivalently,
$|S_{\mathrm{high}}|<k<|S_{\mathrm{high}}|+|S_{\mathrm{tie}}|$), the set $\partial g(x)$ is a
polytope generated by choosing which $n_{\mathrm{pick}}\mathrel{:=}k-|S_{\mathrm{high}}|$
indices inside $S_{\mathrm{tie}}$ join the mandatory indices $S_{\mathrm{high}}$ in the
top-$k$ magnitude set. On $S_{\mathrm{high}}$ and $S_{\mathrm{low}}$ the corresponding
$\xi_i\in\partial g(x)$ are fixed at $\xi_i=\lambda\,\mathrm{sgn}(x_i)$ and $\xi_i=0$,
respectively. In particular, on $S_{\mathrm{low}}$ with $x_i=0$ one has
$\mathrm{dist}_2\bigl(0,(\partial f(x))_i\bigr)=\max\{0,|r_i|-\lambda\}$, whereas on
$S_{\mathrm{low}}$ with $x_i\neq 0$ one has
$\mathrm{dist}_2\bigl(0,(\partial f(x))_i\bigr)=|r_i+\lambda\,\mathrm{sgn}(x_i)|$ because
$(\partial f(x))_i$ is the singleton $\{r_i+\lambda\,\mathrm{sgn}(x_i)\}$ (this is generally
\emph{not} the same as $|r_i|$). On $S_{\mathrm{tie}}$ each admissible vertex uses $p_i\in\{0,1\}$ to indicate
whether coordinate $i$ is included in the top-$k$ support; for each $i\in S_{\mathrm{tie}}$
define the squared coordinate distance if $p_i$ is chosen,
\begin{equation*}
  D_i(p)^2 \mathrel{:=}
  \begin{cases}
    \Bigl(\max\bigl\{\,0,\; |r_i|-\lambda(1-p)\,\bigr\}\Bigr)^2, & x_i = 0,\; p\in\{0,1\},\\[4pt]
    \Bigl(\lambda\,\mathrm{sgn}(x_i)(p-1)-r_i\Bigr)^2, & x_i \neq 0,\; p\in\{0,1\},
  \end{cases}
\end{equation*}
where the lower branch is the squared distance from
$\xi_i=\lambda\,\mathrm{sgn}(x_i)\,p$ to the singleton $(\partial f(x))_i=\{r_i+\lambda\,\mathrm{sgn}(x_i)\}$.
Because $(\partial f(x))_i$ is an interval (or a point) independently across $i$, for any
$\xi\in\partial g(x)$ compatible with tie indicators $(p_i)_{i\in S_{\mathrm{tie}}}$ one has
\begin{equation*}
  \mathrm{dist}_2\!\bigl(\xi,\,\partial f(x)\bigr)^2 \;=\; \sum_{i=1}^n D_i(p_i)^2,
\end{equation*}
with the understanding that $p_i$ is fixed on $S_{\mathrm{high}}\cup S_{\mathrm{low}}$ by the
forced pattern ($p_i=1$ on $S_{\mathrm{high}}$, $p_i=0$ on $S_{\mathrm{low}}$) and only the tied
coordinates vary. Thus maximizing $\mathrm{dist}_2(\xi,\partial f(x))$ over $\xi\in\partial g(x)$
reduces to maximizing this separable squared distance over the tie polytope. The maximizer is
attained at a $0/1$ vertex; it is found by the greedy rule
$M_i\mathrel{:=}D_i(1)^2-D_i(0)^2$, sort $\{M_i\}_{i\in S_{\mathrm{tie}}}$ in decreasing order,
set $p_i=1$ on the first $n_{\mathrm{pick}}$ indices and $p_i=0$ on the remaining tied
indices. With $(p_i)$ fixed, let $\xi\in\partial g(x)$ be the corresponding vertex, set
$\delta_i$ on $S_{\mathrm{tie}}$ to $\sqrt{D_i(p_i)^2}$, combine with the
already-fixed coordinates on $S_{\mathrm{high}}$ and $S_{\mathrm{low}}$, and define
\begin{equation*}
  \mathrm{gap}(x) \mathrel{:=} \sqrt{\sum_{i=1}^n \delta_i^2}
  \;=\; \sup_{\eta\in\partial g(x)} \mathrm{dist}_2\!\bigl(\eta,\,\partial f(x)\bigr).
\end{equation*}
The same pass criterion $\delta_i\le \varepsilon$ coordinatewise with $\varepsilon=10^{-6}$
certifies $\mathrm{gap}(x)=0$ up to floating-point tolerance, hence~\eqref{eq:tl-dstat-dist-sup}.

\paragraph{Least trimmed squares (exact smooth-$f$ DC test).}
Here $h(x)=\tfrac12\|Ax-b\|^2-\tfrac12\,\mathrm{top}_q((Ax-b)^2)=f(x)-g(x)$ with $f$
smooth convex and $g$ convex. In this template, d-stationarity is equivalent to
$g$ being differentiable at $x$ and satisfying $\nabla f(x)=\nabla g(x)$: if $g$ is
not differentiable, $\partial g(x)$ is not a singleton and cannot be contained in
the singleton $\partial f(x)=\{\nabla f(x)\}$. We form squared residuals
$r_i^2$, partition indices into above-, at-, and below-threshold groups relative
to the top-$q$ cutoff, and set $\nabla f(x)=A^\top r$. If there is no residual tie,
the top-$q$ active set is unique, $\nabla g(x)$ is explicit, and the gap is
$\|\nabla f(x)-\nabla g(x)\|_2$. If there are ties, we first test
one admissible top-$q$ completion; if it already violates gradient match beyond
$10^{-6}$, the point is declared non-d-stationary immediately. Otherwise we measure
nondifferentiability by the maximum, over admissible completions $T$ of the tied
block, of $\|\nabla g_T(x)-\nabla g_{T_{\mathrm{ref}}}(x)\|_2$, where
$\nabla g_T(x)=A^\top w$ with $w_i=r_i$ on $T$ and $w_i=0$ off $T$. Exhaustive
enumeration is used when the number of completions is at most $8192$; for larger tie
sets a deterministic pair of extreme completions plus $256$ random completions
(estimating the witness). The acceptance tolerance is $10^{-6}$ on the resulting gap.

\paragraph{ReLU regression (exact DC inclusion check).}
For ReLU regression we now use an explicit DC decomposition and test d-stationarity
through convex-subdifferential inclusion. With
\begin{equation*}
  h(x)=\tfrac12\sum_{i=1}^m\bigl(\max\{0,a_i^\top x\}-b_i\bigr)^2,
\end{equation*}
define $h=f-g$ by
\begin{equation*}
\begin{aligned}
  f(x) &\mathrel{:=} \tfrac12\sum_{i=1}^m\bigl(\max\{0,a_i^\top x\}\bigr)^2
  + \sum_{i=1}^m \alpha_i \max\{0,a_i^\top x\},\\
  g(x) &\mathrel{:=} \sum_{i=1}^m \beta_i \max\{0,a_i^\top x\},
\end{aligned}
\end{equation*}
where $\beta_i \mathrel{:=} b_i$ for $b_i>0$ and $0$ otherwise, and
$\alpha_i \mathrel{:=} -b_i$ for $b_i<0$ and $0$ otherwise. Then
$\alpha_i\beta_i=0$ for all $i$, and d-stationarity is the DC condition
$\partial g(x)\subseteq\partial f(x)$.

Write $z\mathrel{:=}Ax$, and use the numerical kink bandwidth
\begin{equation*}
  \tau_{\mathrm{kink}}(x) \mathrel{:=} 10^{-6}\,\max\bigl\{\,1,\;\|z\|_\infty\,\bigr\}.
\end{equation*}
Set $I_+(x)\mathrel{:=}\{i:z_i>\tau_{\mathrm{kink}}(x)\}$ and
$I_0(x)\mathrel{:=}\{i:|z_i|\le\tau_{\mathrm{kink}}(x)\}$, and split
$I_0(x)=I_{0,\alpha}(x)\sqcup I_{0,\beta}(x)$ with
\begin{equation*}
  I_{0,\alpha}(x)\mathrel{:=}\{i\in I_0(x):\alpha_i>0\},
  \qquad
  I_{0,\beta}(x)\mathrel{:=}\{i\in I_0(x):\beta_i>0\}.
\end{equation*}
Define
\begin{equation*}
\begin{aligned}
  g_0(x) &\mathrel{:=} \sum_{i\in I_+(x)} \beta_i a_i,\\
  f_0(x) &\mathrel{:=} \sum_{i\in I_+(x)} \bigl(z_i+\alpha_i\bigr)a_i,\\
  D(x)   &\mathrel{:=} g_0(x)-f_0(x),
\end{aligned}
\end{equation*}
and generator matrices
\begin{equation*}
  G_\beta(x)\mathrel{:=}\bigl[\beta_i a_i\bigr]_{i\in I_{0,\beta}(x)},
  \qquad
  F_\alpha(x)\mathrel{:=}\bigl[\alpha_i a_i\bigr]_{i\in I_{0,\alpha}(x)}.
\end{equation*}
Then $\partial g(x)\subseteq\partial f(x)$ is equivalent to
\begin{equation*}
  D(x)+G_\beta(x)\lambda \in \bigl\{F_\alpha(x)\mu:\mu\in[0,1]^{|I_{0,\alpha}(x)|}\bigr\}
  \quad \forall \lambda\in[0,1]^{|I_{0,\beta}(x)|}.
\end{equation*}
The implemented diagnostic evaluates the exact vertex form of this condition:
\begin{equation*}
  \Gamma_{\mathrm{ReLU}}(x)\mathrel{:=}
  \max_{\lambda\in\{0,1\}^{|I_{0,\beta}(x)|}}
  \min_{\mu\in[0,1]^{|I_{0,\alpha}(x)|}}
  \bigl\|D(x)+G_\beta(x)\lambda-F_\alpha(x)\mu\bigr\|_2.
\end{equation*}
Each inner problem is a box-constrained least-squares subproblem, and the outer
maximum is over binary vertices. We declare \emph{pass} iff
\begin{equation*}
  \Gamma_{\mathrm{ReLU}}(x)\le 10^{-6}.
\end{equation*}
When $|I_{0,\beta}(x)|$ exceeds the implementation cap (set to $16$), the run is marked
non-d-stationary for reporting rather than using an uncontrolled approximation.

\subsection{Trimmed Lasso: Resolving the Failure Mode}

The trimmed-lasso objective
\[
  h(x)=\tfrac12\|Ax-b\|^2+\lambda\bigl(\|x\|_1-\mathrm{top}_k|x|\bigr)
\]
is the regime where plain DCA is known to converge to critical points that can fail
d-stationarity. We focus on $\lambda=1$ with
$m=50$, $n=100$, $k=5$, noise std.\ $0.1$, and $N=5000$ outer iterations.
Both samplers are run; the aggregate table reports gauss-axis ($\mu=300$).

\begin{center}
\begin{tabular}{l c c c c c c}
\toprule
$\lambda$ & $n$ & win/tie/loss & mean gap & median gap (wins) & McNemar $p$ & \# non-d-stat (base/aug) \\
\midrule
1.00 & 100 & 99/1/0 & +2.94e-01 & +2.69e-01 & 1.58e-30 & 99/4 \\
\bottomrule
\end{tabular}

\end{center}

The central empirical message is clear: DCA fails the d-stationarity diagnostic on
$99/100$ instances, while augmentation with gauss-axis reduces this to $4/100$.
At the objective level, augmentation wins on $99$ instances, ties on one, and never
loses. McNemar's exact one-sided test gives $p\approx 1.6\times 10^{-30}$.

Sphere runs were still checked for completeness, but we do not plot
them here because they add little visual information in this focused regime:
the sphere variant accepts no exploration moves and is practically indistinguishable
from DCA. Instead, we show only the gauss-axis plots in
\Cref{fig:v2-tl-scatter,fig:v2-tl-trajectory}.

The reason this happens is structural. In high dimension, a random sphere direction has
typical $\ell_1$ mass of order $\sqrt{n}$ after $\ell_2$ normalization. For trimmed
lasso, the nonconvex sparsity term is driven by coordinate-level behavior, so this
``diffuse'' direction tends to spread motion over many coordinates and rarely aligns with
the sparse descent geometry near stuck points. The acceptance test then almost never
fires in finite budgets. The axis-biased Gaussian mixture is designed to correct exactly
this mismatch: each sampled direction is still absolutely continuous on the sphere and
keeps full support (so the theorem assumptions are unchanged), but it places much larger
probability on directions close to coordinate axes, which are precisely the directions
that can unlock descent for this $\ell_1$-structured objective. Empirically, this
sampler-design change is the difference between inert exploration and systematic
improvement.

\begin{figure}[t]
\centering
\includegraphics[width=0.58\linewidth]{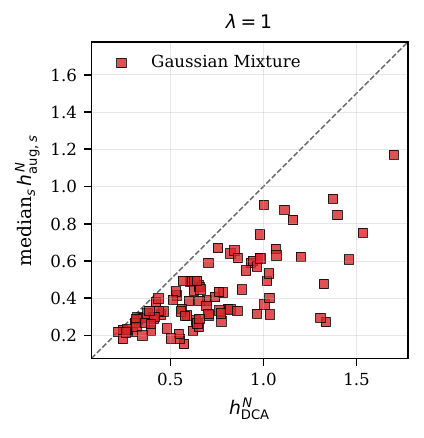}
\caption{Trimmed lasso scatter at $\lambda=1$. Horizontal axis: final DCA objective.
Vertical axis: median augmented objective over three seeds. Points below the diagonal
favor augmentation. (Sphere sampler omitted from display: in this regime it produces no
accepted exploration moves and overlays DCA.)}
\label{fig:v2-tl-scatter}
\end{figure}

\begin{figure}[t]
\centering
\includegraphics[width=0.62\linewidth]{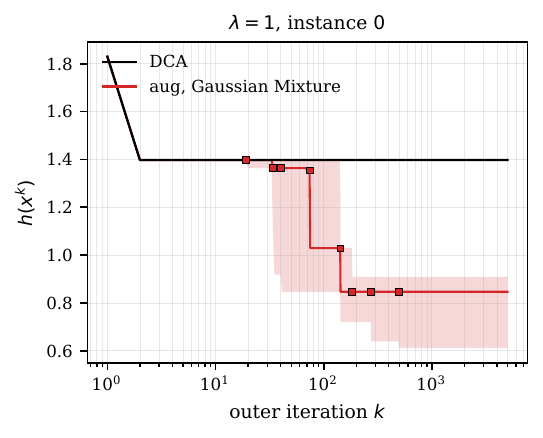}
\caption{Trimmed lasso trajectory on a representative instance. Black: DCA.
Colored band/curve: min/median/max over augmented seeds. The Gaussian-mixture sampler
breaks away from the DCA plateau. (Sphere sampler omitted from display: trajectory
essentially overlays DCA.)}
\label{fig:v2-tl-trajectory}
\end{figure}

\subsection{Least Trimmed Squares: Non-Disruptive Behavior with Occasional Escapes}

For least trimmed squares,
\[
  h(x)=\tfrac12\|Ax-b\|^2-\tfrac12\,\mathrm{top}_q((Ax-b)^2),
\]
we focus on $\sigma_{\mathrm{clean}}=4$ with
$m=100$, $n=50$, $q=10$, outlier std.\ $10$, and $N=5000$.
For the focused LTS report, exploration directions are sampled from the sphere
distribution $v\sim \mathrm{Unif}(\mathbb{S}^{n-1})$.

\begin{center}
\begin{tabular}{l c c c c c c}
\toprule
$\sigma_{\mathrm{clean}}$ & $n$ & win/tie/loss & mean gap & median gap (wins) & McNemar $p$ & \# non-d-stat (base/aug) \\
\midrule
4.00 & 100 & 9/91/0 & +3.96e+00 & +5.22e+01 & 1.95e-03 & 0/0 \\
\bottomrule
\end{tabular}

\end{center}

This experiment supports a different part of the story. Here the base DCA oracle is
already reliable in first-order terms: both base and augmented runs have zero
d-stationarity failures in the focused setting. Augmentation is therefore expected to be
mostly neutral, and that is exactly what we observe ($91$ ties). Yet on the discordant
subset, all outcomes favor augmentation ($9$ wins, $0$ losses), yielding
$p\approx 2.0\times 10^{-3}$ by McNemar's exact test. In short, augmentation is
non-disruptive where the oracle already behaves well, but can still find a better basin
when one is nearby.

\begin{figure}[t]
\centering
\includegraphics[width=0.45\linewidth]{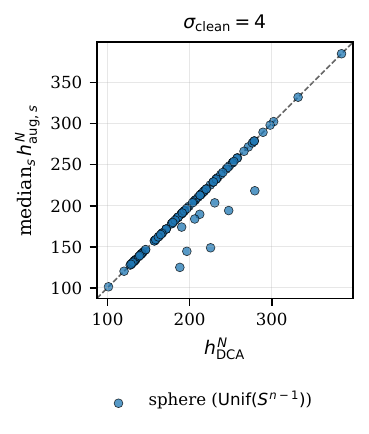}
\caption{LTS scatter at $\sigma_{\mathrm{clean}}=4$ (sphere reporting).
Most points lie on the diagonal (ties), with a smaller set below (augmentation wins).}
\label{fig:v2-lts-scatter}
\end{figure}

\begin{figure}[t]
\centering
\includegraphics[width=0.6\linewidth]{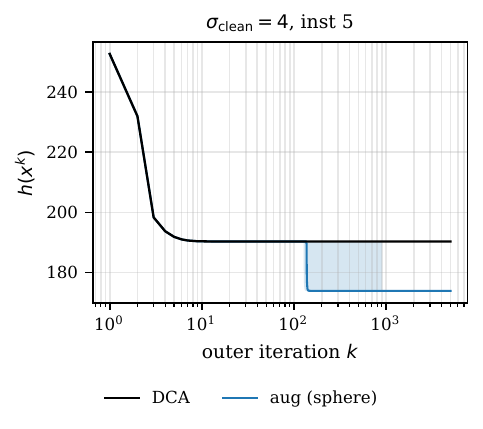}
\caption{LTS representative trajectory where augmentation improves over the DCA plateau.}
\label{fig:v2-lts-trajectory}
\end{figure}

\subsection{ReLU Regression: Transfer Beyond DC Oracles}

The ReLU experiment is intentionally outside the DC setup:
\[
  h(x)=\tfrac12\sum_i\bigl(\mathrm{relu}(a_i^\top x)-b_i\bigr)^2,
\]
optimized by a prox-linear oracle ($\rho_{\mathrm{prox}}=0.1$), then augmented by the
same exploration mechanism used elsewhere. We aggregate two focused settings,
$(q,\rho_b)\in\{(0.2,2),(0.4,2)\}$, with $100$ instances per setting
($200$ paired comparisons total), $m=200$, $n=50$, and $N_{\mathrm{outer}}=1000$.
For the focused ReLU report, exploration directions are likewise sampled as
$v\sim \mathrm{Unif}(\mathbb{S}^{n-1})$.

\begin{center}
\begin{tabular}{l c c c c c c}
\toprule
setting & $n$ & win/tie/loss & mean gap & median gap (wins) & McNemar $p$ & \# non-d-stat (base/aug) \\
\midrule
focused sweep & 200 & 44/156/0 & +9.58e-01 & +1.26e+00 & 5.68e-14 & 0/0 \\
\bottomrule
\end{tabular}

\end{center}

As in LTS, augmentation is largely neutral but never harmful in the focused report:
$44$ wins, $156$ ties, $0$ losses, with zero d-stationarity failures for both base and
augmented methods. McNemar's exact one-sided test gives
$p\approx 5.7\times 10^{-14}$. Since the inner oracle here is prox-linear (not DCA), this
result is direct empirical evidence for the paper's oracle-agnostic claim.

\begin{figure}[t]
\centering
\includegraphics[width=0.42\linewidth]{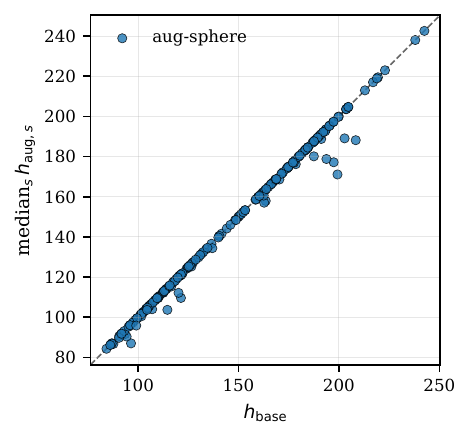}
\caption{ReLU scatter (sphere reporting). Most pairs tie, with a smaller strict-win set.}
\label{fig:v2-relu-scatter}
\end{figure}

\begin{figure}[t]
\centering
\includegraphics[width=0.55\linewidth]{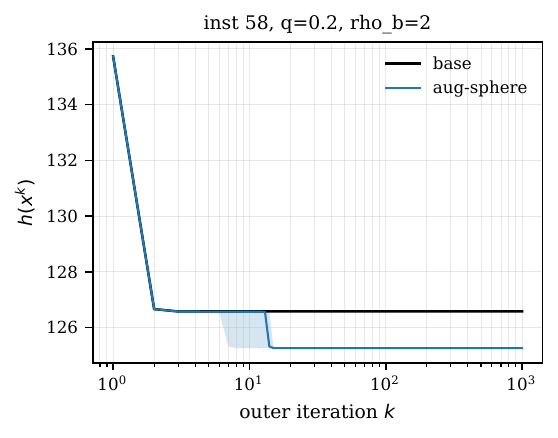}
\caption{ReLU representative trajectory (base prox-linear vs.\ augmented seeds).}
\label{fig:v2-relu-trajectory}
\end{figure}

\subsection{Cross-Experiment Interpretation}

Taken together, the three experiments show a consistent pattern with distinct roles.
Trimmed lasso demonstrates the failure mode and its correction: augmentation converts a
high non-d-stationarity regime into a mostly d-stationary one, with strong objective
improvements. LTS and ReLU demonstrate the complementary regime: when the base oracle is
already first-order sound, augmentation is mostly neutral, never harmful in the focused
runs, and occasionally beneficial.

This is the practical reading of the theoretical separation in the paper.
Oracle choice controls local descent structure and rates; exploration controls asymptotic
d-stationarity behavior in an oracle-agnostic way. The experiments were designed so that
each benchmark contributes one part of that statement, and the combined suite supports it
without requiring any algorithm-specific retuning of the exploration step.

\newpage
\section*{NeurIPS Paper Checklist}

\begin{enumerate}

\item {\bf Claims}
    \item[] Question: Do the main claims made in the abstract and introduction accurately reflect the paper's contributions and scope?
    \item[] Answer: \answerYes{}
    \item[] Justification: The abstract and introduction explicitly state the proposed random exploration augmentation, the almost-sure d-stationarity guarantee, and the accompanying rate-preservation claims. The scope is limited to the assumptions and settings formalized in the manuscript.

\item {\bf Limitations}
    \item[] Question: Does the paper discuss the limitations of the work performed by the authors?
    \item[] Answer: \answerYes{}
    \item[] Justification: We have included a ``Limitations'' section which discusses the limitation of our work.

\item {\bf Theory assumptions and proofs}
    \item[] Question: For each theoretical result, does the paper provide the full set of assumptions and a complete (and correct) proof?
    \item[] Answer: \answerYes{}
    \item[] Justification: The draft states assumptions A1--A3 and provides full theorem statements and proofs in the main text for the central convergence results. Lemmas and theorems are numbered and cross-referenced.

\item {\bf Experimental result reproducibility}
    \item[] Question: Does the paper fully disclose all the information needed to reproduce the main experimental results of the paper to the extent that it affects the main claims and/or conclusions of the paper (regardless of whether the code and data are provided or not)?
    \item[] Answer: \answerYes{}
    \item[] Justification: The experimental setting is described in detail in \cref{sec:appendix-experiments-v2}.

\item {\bf Open access to data and code}
    \item[] Question: Does the paper provide open access to the data and code, with sufficient instructions to faithfully reproduce the main experimental results, as described in supplemental material?
    \item[] Answer: \answerNo{}
    \item[] Justification: The manuscript's central contribution is theoretical and does not hinge on empirical code/data artifacts. This checklist item is therefore not applicable to the main claims.

\item {\bf Experimental setting/details}
    \item[] Question: Does the paper specify all the training and test details (e.g., data splits, hyperparameters, how they were chosen, type of optimizer) necessary to understand the results?
    \item[] Answer: \answerYes{}
    \item[] Justification: The full details of the experimental setup are in \cref{sec:appendix-experiments-v2}.

\item {\bf Experiment statistical significance}
    \item[] Question: Does the paper report error bars suitably and correctly defined or other appropriate information about the statistical significance of the experiments?
    \item[] Answer: \answerYes{}
    \item[] Justification: The paper provides statistical significance results using McNemar's test in \cref{sec:appendix-experiments-v2}.

\item {\bf Experiments compute resources}
    \item[] Question: For each experiment, does the paper provide sufficient information on the computer resources (type of compute workers, memory, time of execution) needed to reproduce the experiments?
    \item[] Answer: \answerYes{}
    \item[] Justification: The experiments were run on a laptop (CPU setting) and each focused run completed within several minutes; this compute scale is sufficient for reproducing the reported results.

\item {\bf Code of ethics}
    \item[] Question: Does the research conducted in the paper conform, in every respect, with the NeurIPS Code of Ethics \url{https://neurips.cc/public/EthicsGuidelines}?
    \item[] Answer: \answerYes{}
    \item[] Justification: The work is methodological optimization research and does not involve human participants or sensitive personal data. The authors commit to complying with NeurIPS ethical standards in preparation, reporting, and release practices.

\item {\bf Broader impacts}
    \item[] Question: Does the paper discuss both potential positive societal impacts and negative societal impacts of the work performed?
    \item[] Answer: \answerNA{}
    \item[] Justification: The submission is foundational optimization theory without a direct deployment pathway discussed in the current draft. A broader-impact discussion can be added in a revised version if required by venue policy.

\item {\bf Safeguards}
    \item[] Question: Does the paper describe safeguards that have been put in place for responsible release of data or models that have a high risk for misuse (e.g., pre-trained language models, image generators, or scraped datasets)?
    \item[] Answer: \answerNA{}
    \item[] Justification: The manuscript does not release high-risk models or scraped datasets. Hence safeguards for such releases are not applicable.

\item {\bf Licenses for existing assets}
    \item[] Question: Are the creators or original owners of assets (e.g., code, data, models), used in the paper, properly credited and are the license and terms of use explicitly mentioned and properly respected?
    \item[] Answer: \answerNA{}
    \item[] Justification: No third-party data/model asset release is central to the paper's theoretical claims in the current submission version. Therefore this item is not applicable.

\item {\bf New assets}
    \item[] Question: Are new assets introduced in the paper well documented and is the documentation provided alongside the assets?
    \item[] Answer: \answerNA{}
    \item[] Justification: The current draft does not introduce or release a new dataset/model asset as a core contribution. This question is therefore not applicable.

\item {\bf Crowdsourcing and research with human subjects}
    \item[] Question: For crowdsourcing experiments and research with human subjects, does the paper include the full text of instructions given to participants and screenshots, if applicable, as well as details about compensation (if any)?
    \item[] Answer: \answerNA{}
    \item[] Justification: The work does not involve crowdsourcing or human-subject studies. As such, this item does not apply.

\item {\bf Institutional review board (IRB) approvals or equivalent for research with human subjects}
    \item[] Question: Does the paper describe potential risks incurred by study participants, whether such risks were disclosed to the subjects, and whether Institutional Review Board (IRB) approvals (or an equivalent approval/review based on the requirements of your country or institution) were obtained?
    \item[] Answer: \answerNA{}
    \item[] Justification: No human-subject data collection or experimentation is conducted in this research. Therefore IRB-related disclosure is not applicable.

\item {\bf Declaration of LLM usage}
    \item[] Question: Does the paper describe the usage of LLMs if it is an important, original, or non-standard component of the core methods in this research? Note that if the LLM is used only for writing, editing, or formatting purposes and does \emph{not} impact the core methodology, scientific rigor, or originality of the research, declaration is not required.
    \item[] Answer: \answerNA{}
    \item[] Justification: LLMs are not part of the proposed optimization method, theoretical analysis, or experimental methodology. Consequently, declaration is not required under the stated policy.

\end{enumerate}

\end{document}